\documentclass[11pt]{article}

\usepackage{mathptmx}      
\usepackage{tgtermes}
\usepackage{mathtools}
\usepackage{amsmath,amssymb,amsthm}
\usepackage{bbm}
\usepackage[title]{appendix}
\usepackage{xcolor}
\usepackage{comment}
\usepackage{subcaption}
\usepackage{fancyhdr}
\newcommand{\Rb}{\mathbbm{R}}      
\newcommand{\Pb}{\mathbbm{P}}
\newcommand{\Ac}{\mathcal{A}}

\newcommand{\Dc}{\mathcal{D}}
\newcommand{\Qc}{\mathcal{Q}}
\newcommand{\Bb}{\mathbbm{B}}
\newcommand{\Eb}{\mathbbm{E}}
\newcommand{\Nb}{\mathbbm{N}}
\newcommand{\Zb}{\mathbbm{Z}}

\newcommand{\Fc}{\mathcal{F}}

\newcommand{\Kc}{\mathcal{K}}

\newcommand{\Nc}{\mathcal{N}}
\newcommand{\Pc}{\mathcal{P}}
\newcommand{\Sc}{\mathcal{S}}

\newcommand{\Vc}{\mathcal{V}}
\newcommand{\Xc}{\mathcal{X}}

\newcommand{\Tc}{\mathcal{T}}

\newcommand{\1}{\mathbbm{1}}
\newcommand{\mub}{\mathbbm{\mu}}
\newcommand{\nub}{\mathbbm{\nu}}

\newcommand{\argmin}{\mathop{\rm argmin}}

\newcommand{\dist}{\mathop{\rm dist}}

\newcommand{\D}{\mathop{\text{\rm d}\!}}

\newcommand{\Cdot}{\,\cdot\,}

\newcommand*{\dt}[1]{\overset{\hbox{\tiny${\;\,}_\bullet$}}{#1}}

\newcommand*\samethanks[1][\value{footnote}]{\footnotemark[#1]}

\newtheorem{theorem}{Theorem}[section]
\newtheorem{proposition}[theorem]{Proposition}
\newtheorem{lemma}[theorem]{Lemma}
\newtheorem{example}[theorem]{Example}
\newtheorem{definition}[theorem]{Definition}
\newtheorem{corollary}[theorem]{Corollary}
\newtheorem{remark}[theorem]{Remark}

\newenvironment{tightlist}[1]{%
    \list{{\textup{(\roman{enumi})}}}{\settowidth\labelwidth{{\textup{(#1)}}}
    \leftmargin 0pt \advance\leftmargin\labelsep \itemindent \parindent
    \parsep 0pt plus 1pt minus 1pt \topsep 0pt \itemsep 0pt
    \usecounter{enumi}}}{\endlist}

\definecolor{dblue}{rgb}{0,0,0.7}
\definecolor{plum}{rgb}{0.3,0,0.7}

 \makeatletter
\newcommand*\rel@kern[1]{\kern#1\dimexpr\macc@kerna}
\newcommand*\widebar[1]{%
  \begingroup
  \def\mathaccent##1##2{%
    \rel@kern{0.8}%
    \overline{\rel@kern{-0.8}\macc@nucleus\rel@kern{0.2}}%
    \rel@kern{-0.2}%
  }%
  \macc@depth\@ne
  \let\math@bgroup\@empty \let\math@egroup\macc@set@skewchar
  \mathsurround\z@ \frozen@everymath{\mathgroup\macc@group\relax}%
  \macc@set@skewchar\relax
  \let\mathaccentV\macc@nested@a
  \macc@nested@a\relax111{#1}%
  \endgroup
}
\makeatother

\title{A Functional Model Method for Nonconvex Nonsmooth Conditional Stochastic Optimization\thanks{Accepted for publication in \emph{SIAM Journal on Optimization}. This work was supported by the Office of Naval Research Award N00014-21-1-2161.}}

\author{
Andrzej Ruszczy\'{n}ski\thanks{Rutgers University,
Department of Management Science and Information Systems,
 94 Rockefeller Rd, Piscataway, NJ 08854, USA ({\tt rusz@rutgers.edu; shangzhe.yang@rutgers.edu}).}
 \and
 Shangzhe Yang\samethanks
 }

\begin{document}

\maketitle

\begin{abstract}
We consider stochastic optimization problems involving an expected value of a nonlinear function of a base random vector and a conditional expectation of another function depending on the base random vector, a dependent random vector, and the decision variables. We call such problems conditional stochastic optimization problems. They arise in many applications, such as uplift modeling, reinforcement learning, and contextual optimization.
We propose a specialized single time-scale stochastic method for nonconvex constrained conditional stochastic optimization problems with a Lipschitz smooth outer function and a generalized differentiable inner function. In the method, we approximate the inner conditional expectation with a rich parametric model whose mean squared error satisfies a stochastic version of a {\L}ojasiewicz condition. The model is used by an inner learning algorithm.
 The main feature of our approach is that unbiased stochastic estimates of the directions used by the method can be generated with one observation from the joint distribution per iteration, which makes it applicable to real-time learning. The directions, however, are not
 gradients or subgradients of any overall objective function. We prove the convergence of the method with probability one, using the method of differential inclusions and a specially designed Lyapunov function, involving a stochastic generalization of the Bregman distance. Finally,  a numerical illustration demonstrates the viability of our approach.\\
 \emph{Keywords:}
 Conditional Stochastic Optimization, Nonsmooth Optimization, Stochastic Subgradient Methods, Reparametrization.\\
 \emph{AMS:}
  90C15,  49J52, 60-08.
\end{abstract}

\section{Introduction}
\label{s:1}
The \emph{Conditional Stochastic Optimization} (CSO) problem is formulated as follows:
\begin{equation}
\min_{\beta \in {\Bb}} G(\beta) \triangleq \Eb \Big\{g\big(\Eb [f(X,Y,\beta) \,|\, X]\big)\Big\}
\label{cso}
\end{equation}
where $\Bb \subset \Rb^{n_\beta}$, $f: \Rb^{n_X} \times \Rb^{n_Y} \times \Rb^{n_\beta} \to \Rb^{n_f}$ , and $g: \Rb^{n_f} \to \Rb$. In formula \eqref{cso}, $X$ and $Y$ are random variables, and the symbols $\Eb[\Cdot]$ and $\Eb[\Cdot|X]$ represent the expected value and the conditional expected value, respectively.

Some important applications can be cast in this form, as illustrated in the examples below.

\begin{example}[\textit{Reinforcement Learning}]
\label{e:RL}
{\rm
Consider an infinite-horizon discounted Ma\-rkov Decision Process $\{\Sc, \Ac, \Pb, r, \gamma \}$ where $\Sc$ is a state space, $\Ac$ is an action space, $\Pb:\Sc\times \Ac \to \Pc(\Sc)$ is a controlled transition kernel, with $\Pc(\Sc)$ denoting the space of probability measures on the space $\Sc$, $r: \Sc \times \Ac \times \Sc \to \Rb$ is a reward function with $r(s, a, s')$ representing the immediate reward after transition from $s$ to $s'$ with action $a$, and $\gamma \in (0,1)$ is a discount factor. Suppose a kernel $\varPi:\Sc \to \Pc(\Ac)$ represents a stationary randomized Markov policy. Our goal is to evaluate the policy $\varPi$, that is, to find the function:
\[
Q^\varPi(s, a) = \Eb\bigg[\sum_{t=0}^\infty \gamma^t r(s_t, a_t, s_{t+1}) \,\bigg|\, s_0 = s, a_0 = a\bigg], \quad (s, a) \in \Sc \times \Ac,
\]
where the random sequence of state--control pairs $(s_t,a_t)$ is generated by the process: $a_t \sim \varPi(s_t)$,
$s_{t+1} \sim \Pb^{a_t}(s_t)$. Here, $\Pb^a(s)$ represents the conditional distribution of the next state
when the present state is $s$ and the action is $a$.
The function $Q^\varPi(\Cdot, \Cdot)$ is the solution of the following equation:
\[
Q^\varPi(s, a) = \Eb_{{s'\sim \Pb^{a}(s)},\,{a' \sim \varPi(s')}} \big[ r(s, a, s') + \gamma \, Q^\varPi(s', a')\big], \quad (s, a) \in \Sc \times \Ac.
\]
Except for small-size problems, this functional equation  cannot be solved exactly, because it
is impossible to visit all state--action pairs.
Therefore, we resort to minimizing
the \textit{mean-squared Bellman error} (see \cite{dann2014policy} and the references therein):
\[
\min_{Q \in \Qc} \Eb_{{s\sim q^\varPi},\,{ a\sim \varPi(s)}}\Big[\big(Q(s, a) - \Eb_{s'\sim \Pb^a(s), a'\sim \varPi(s)}[r(s, a, s') + \gamma \,Q(s', a')]\big)^2\Big],
\]
where $q^\varPi$ is the stationary state distribution under the policy $\varPi$, and $\Qc$
is a class of functions in which we seek a good approximation of the policy value.
An established approach is to use
$\Qc = \{\varPhi(\Cdot, \Cdot, \beta): \beta \in {\Bb}\}$, with  a parametric family of functions $\varPhi:\Sc \times \Ac \times \Bb \to \Rb$.
The resulting problem can be treated as a CSO problem, where $X = \{s, a\}$, $Y = \{s', a'\}$, and:
\begin{equation}
    \label{f-g-MDP}
\begin{aligned}
f(X, Y, \beta) &= \varPhi(s, a,\beta) - r(s, a, s') - \gamma \,\varPhi(s', a', \beta), \\
g(u) &= u^2.
\end{aligned}
\end{equation}
The fundamental difficulty in {high-dimensional learning problems of this type} is that we can observe the triplets $(s,a,s')$, but we have no access to repeated observations of $s'$, for the same values of $s$ and $a$, unless we know the system's model and use a simulator.}
\end{example}

\begin{example}[\textit{Uplift Modeling}]
\label{e:uplift}
{\rm
 Suppose $X$ represents the (random) features of an individual, $T \in \{-1, 1\}$ is a random variable indicating whether the treatment has been applied to this individual, and a real random variable $Y$ represents the outcome observed.
The goal is to estimate the \emph{individual uplift} or \emph{individual treatment effect}
\[
u(x)\triangleq \Eb[Y|X=x,\,T=1] - \Eb[Y|X=x,\,T=-1].
\]
In the causal inference literature, this problem has been widely used for modeling personalized treatment and targeted marketing (see, e.g., \cite{jaskowski2012uplift,yamane2018uplift}).
Again, while the triplets $(X,T,Y)$ can be observed, we cannot observe two outcomes  { $Y_1(x)$ and $Y_{-1}(x)$} for the same
individual $X=x$, with and without treatment.

Suppose we know the individual treatment policy $\varPi(x) := \Pb[T=1|X=x]\in (0,1)$. We define an auxiliary random variable $Z$:
\[
Z(X,T,Y) =
\begin{cases}
\frac{Y}{\varPi(X)} & \text{ if } T = 1,\\
-\frac{Y}{1-\varPi(X)} & \text{ if } T = -1.
\end{cases}
\]
Then
\begin{align*}
\Eb[Z|X=x]
&= \Eb[Z|X=x, T=1] \varPi(x) + \Eb[Z|X=x, T=-1] (1-\varPi(x)) \\
&= \Eb[Y|X=x,T=1] - \Eb[Y|X=x,T=-1] = u(x).
\end{align*}
If we choose a model class $u(\Cdot) \approx \varPhi(\Cdot, \beta)$,
with some parameters $\beta \in {\Bb}$, then we can formulate the following problem:
\[
\min_{\beta \in {\Bb}} \Eb\big\{g\big(\varPhi(X, \beta), \Eb[Z|X]\big)\big\}.
\]
Here $g(\Cdot, \Cdot)$ may be an arbitrary loss function measuring the discrepancy between its arguments.
The last problem is a CSO problem, where $Z$ plays the role of $Y$,  and $f = [f_1, f_2]^T$, with
\begin{align*}
f_1(X,Z,  \beta) &= \varPhi(X, \beta), \\
f_2(X, Z, \beta) &= Z.
\end{align*}
}
\end{example}
\begin{example}[\textit{Contextual Optimization}]
\label{e:context}
{\rm
Suppose a \textit{context variable} $X$ is observed before a decision is made
in a stochastic optimization problem involving a random variable $Y$ dependent on $X$ (see \cite{kannan2022data,sadana2023survey} and the references therein).
Denoting by $u$ the
 decision variable, the contextual optimization problem is to find a \textit{decision rule} $u(X)$ that minimizes
the conditional cost $\Eb[ c(u(X),X,Y)|X]$, for all possible values of the context $X$. Here, $c:\Rb^{n_u}\times\Rb^{n_x}\times\Rb^{n_y}\to \Rb$ is a known cost function. If we
restrict the decision rules to a functional class $u(\Cdot)=\varPi(\Cdot,\beta)$, with parameters $\beta\in {\Bb}$, and average the cost over all contexts, we can formulate the following parametric policy optimization problem:
\[
\min_{\beta\in {\Bb}} \Eb\Big\{\Eb\big[ c(\varPi(X,\beta),X,Y)\big|X\big]\Big\} = \min_{\beta\in {\Bb}} \Eb\Big\{ c(\varPi(X,\beta),X,Y)\Big\}.
\]
It is a standard expected value optimization problem. A different situation occurs when we have a \textit{benchmark function} $b:\Rb^{n_x}\to \Rb$ and relate the conditional cost $\Eb[ c(u,X,Y)|X]$ to $b(X)$. For example, we may
use a convex increasing function $\ell:\Rb\to\Rb$ and formulate the problem
\[
\min_{\beta\in {\Bb}} \Eb\Big\{\ell\Big(\Eb\big[ c(\varPi(X,\beta),X,Y)\big|X\big]- b(X)\Big)\Big\}.
\]
It is a special case of problem \eqref{cso}.
}
\end{example}

The model \eqref{cso} is related to the composition optimization problem
\begin{equation}
\min_{\beta \in {\Bb}}  \Eb \Big\{g\big(X,\Eb [f(Y,\beta)]\big)\Big\},
\label{comp}
\end{equation}
which uses the complete expected value as an argument of the function $g(\Cdot)$. Problems
of this type were considered already in
 \cite{ermol1971general} and \cite[Ch. V.4]{ermoliev1976methods}, and have been recently revisited due to
 their relevance in machine learning
\cite{wang2017stochastic,WaLiFa17}.
The first single time-scale method for problem \eqref{comp} with continuously differentiable functions has been proposed in \cite{ghadimi2018single}, and further generalized to nonsmooth functions and multiple levels in \cite{ruszczynski2021stochastic}. Rates for smooth problems were provided in \cite{balasubramanian2022stochastic}.

The difference between \eqref{cso} and \eqref{comp} is that
the latter allows for treating $\Eb [f(Y,\beta)]$ as a function of $\beta$. Thanks to the availability of its value and (sub)gradient estimates, one can devise ``tracking'' algorithms (filters), that
estimate its values at the sequence $\{\beta^k\}$ generated by the method. In fact, the existing methods for \eqref{comp} exploit this idea, at different levels of sophistication.
In \eqref{cso}, the inner conditional expectation is a function of $X$ and $\beta$; therefore, ``tracking'' in a functional space is needed. It is extremely difficult to construct a statistical estimate of the objective value in \eqref{cso} or its (sub)gradient,
if only sampling from the joint distribution of $(X,Y)$ is possible.

Ref. \cite{hu2020sample} analyses a plug-in approximation method for solving \eqref{cso}.  Assuming the availability of samples from the distribution of $X$ and the conditional
distribution of $Y$, given $X$, the authors establish the sample complexity of the sample
average approximation procedure.  In \cite{hu2020biased} a biased stochastic gradient descent method is proposed for \eqref{cso}, employing a simulator that allows for generating
samples from the conditional expectation of increasing sizes. This allows for constructing better and better stochastic gradient estimates of the composition.
{  Ref. \cite{hu2021bias} reduced the expected number of samples from the conditional distribution
needed to build a low-bias gradient estimator from polynomial to $O(1)$.}
{  Ref. \cite{he2024debiasing} proposed to conduct bias correction via extrapolation. Conditions required for constructing unbiased gradient estimators for CSO problems were introduced in \cite{goda2023constructing}. The application of these techniques is restricted to smooth functions. }

Refs. \cite{dai2017learning,singh2019kernel} consider a related problem of learning the conditional expectation in a functional space. They assume a convex function $g(\Cdot)$, employ a Reproducing Kernel Hilbert Space to represent the conditional expectation function, {  exploit Fenchel duality towards a min-max reformulation, and propose the use of a saddle-point-seeking method. In this context, they do not require a two-level (nested) sampling, but only samples from the joint distribution. This approach is restricted to settings in which a convex-concave problem arises after reparametrization.}

{
It is also worth mentioning that although in some cases $X$ is a discrete random variable, for large sample space, { it is very difficult to have enough repeated observations $(\tilde{X}_i, \tilde{Y}_i)$ for each specific realization $\tilde{X}_i=x$,} due to limited datasets. This results in high bias and variance in the empirical measure.
}

Our main goal is to solve the CSO problem \eqref{cso} by a stochastic algorithm using one observation of the $(X,Y)$-pair per iteration.
Examples \ref{e:RL}--\ref{e:context} illustrate the need for such an approach. To achieve that, we use an auxiliary parametric functional model
of the conditional expectation:
\begin{equation}
    \label{def:F}
F(X, \beta) \triangleq \Eb[f(X, Y, \beta) | X].
\end{equation}
We assume that a sufficiently rich class of functions $\varPsi: \Rb^{n_x} \times \Rb^{n_\theta} \to \Rb^{n_f}$ has been established,
such that for every
$\beta\in {\Bb}$ an (unknown) value of the parameters $\bar{\theta}(\beta) \in \Rb^{n_\theta}$ exists for which
\begin{equation}
\label{tracking-exact}
F(X,\beta) = \varPsi\big(X,\bar{\theta}(\beta)\big)  \quad \text{a.s.}
\end{equation}
Assumptions of this nature are frequently used in the theoretical analysis of machine learning methods,   such as \cite{Bradtke1996,feng2020high,hao2020adaptive,Tsitsiklis1997,zhou2020neural},   to facilitate
\emph{generalization}.
We do not assume that we know the function $\bar{\theta}(\beta)$,
and we do not attempt to calculate its value for any $\beta$;
we merely assume that it exists. This will help to design a ``functional space tracking'' procedure in an algorithm for
solving \eqref{cso}.

One class of models $\varPsi(\Cdot,\Cdot)$ are \emph{linear architecture models},
where we identify \emph{features} $\psi_j(x)$, $j=1,\dots,n_\theta$, and use a regression model
\begin{equation}
    \label{regression}
\varPsi(x,\theta) = \sum_{j=1}^{n_\theta} \theta_j \psi_j(x).
\end{equation}
A more complicated model is a neural network,
in which the features are transformed by a composition of highly nonlinear and possibly nondifferentiable operators
to produce the model $\varPsi(x,\theta)$; the vector $\theta$ represents then various parameters (gains) of the network
\cite{cai2019neural}.

Our approach to \eqref{cso} is based on using the model $\varPsi(X,\theta)$
whenever an estimate of the ``invisible'' conditional expectation $F(X,\beta)$ is needed. As the proper value of $\theta$ is not known, we employ
 the {mean squared error} of the auxiliary model \eqref{tracking-exact}:
\begin{equation}
Q(\beta, \theta) \triangleq \frac{1}{2}
\Eb\big[\|F(X, \beta) - \varPsi(X, \theta)\|^2\big],
\label{te:Q}
\end{equation}
to guide the updates of $\theta$, when $\beta$ is changed.

{  However, we do not plan
to use a brute force two-level approach of the form: keep $\beta$ constant and learn the proper value of $\theta$, then make a step in $\beta$, learn the new $\theta$, and iterate in this way, until convergence. This is the main premise of the \textit{bilevel} optimization settings of the form:
\begin{equation}
\label{bilevel}
\begin{aligned}
\min_{\beta \in \Bb} \quad & \Eb\Big[g\big(\varPsi\big(X, \bar{\theta}(\beta)\big)\big)\Big] \\
\textrm{s.t.} \quad & \bar{\theta}(\beta) = \argmin_{\theta} Q(\beta, \theta);
\end{aligned}
\end{equation}
see, e.g., \cite{chen2021closing}.
We also do not want to use a two-timescale approach, in which the
updates of $\beta$ become infinitely slower than those of $\theta$.
We plan to design a single-timescale method that simultaneously updates $\beta$ and $\theta$
after each observation with stepsizes of the same order, and achieves the tracking accuracy in the limit only.
{ This is in contrast to \cite{schulman2015high}, where multiple episodes are observed between each update, in a heuristic  reinforcement learning method with
parametric value function approximation.}

In the analysis of  the method, we use yet another function:
\begin{equation}
\label{Bregman-stoch}
\Delta^{0}(\beta,\theta) \triangleq \Eb\big[g(F) - g(\varPsi) - \langle \nabla g(\varPsi), F - \varPsi \rangle \big],
\end{equation}
which is a  generalization of the \textit{Bregman distance} \cite{bregman1967relaxation} widely used
in deterministic nonsmooth convex optimization (see \cite{bauschke2001joint} and the references therein). The function $g(\Cdot)$ is assumed to be continuously differentiable but may be non-convex, and thus we use
the distance \eqref{Bregman-stoch} in conjunction with \eqref{te:Q}, to demonstrate
convergence to a critical point of \eqref{cso}.

We assume that the functions $f(x,y,\Cdot)$
and $\varPsi(x,\Cdot)$ are differentiable in a generalized (Norkin) sense, as defined in \cite{norkin1980generalized}. This class of functions is a subset of
locally Lipschitz continuous functions and includes all semismooth functions; in particular, all
subdifferentially regular functions, all weakly convex or weakly concave functions, and their compositions. In Appendix A, we review
the fundamental properties of such functions. One essential feature of this theory is that the application
of standard calculus rules, used by automatic differentiation software, produces valid generalized subgradients and generalized stochastic
subgradients.

The article is organized as follows. In \S \ref{s:2}, we present the relevant assumptions and derive some fundamental properties of the problem.
In \S \ref{s:3} we introduce a single time-scale stochastic subgradient method for solving problem \eqref{cso}, which uses only samples from the joint distribution.
\S \ref{s:4} provides the convergence analysis of the method. Numerical illustration is presented in \S \ref{s:5}.
In the Appendix, we summarize the most important properties of generalized differentiable functions.

Below we review some notation and conventions used throughout the paper. For simplicity, we ignore all the arguments of the functions,
if they are obvious from the context. We also indicate generalized derivatives in the subscript of a function symbol,
that is, $F_x(\Cdot) \in \partial F(\Cdot)$.  We only use the $l_2$-norm in this paper, which is denoted as $\|\Cdot\|$. A dot
over a symbol, for example, $\dt{F}$, represents the generalized derivative of $F(\beta(t))$ with respect to the
time variable, where $\beta(t)$ is an absolutely continuous path (usually, evident from the context).
A function $g: \Rb^n \to \Rb^m$ is \emph{$L$-Lipschitz continuous} if $L > 0$ exists  such that $\|g(x) - g(y)\| \le L \|x-y\|$ for all $x, y \in \Rb^n$. A differentiable function $g: \Rb^n \to \Rb^m$ is called \emph{$L$-Lipschitz smooth} if the vector function $\nabla g(\Cdot)$ is \emph{$L$-Lipschitz continuous}. This implies that $\|g(x) - g(y) - \nabla g(y)^T (x - y)\| \le \frac{L}{2} \|x-y\|^2$ for all $x, y \in \Xc$. A differentiable function $g: \Xc \to \Rb$ is
\emph{strongly convex with modulus $L>0$} if $g(x) - g(y) - \nabla g(y)^T (x - y) \ge \frac{L}{2} \|x-y\|^2$ for all $x, y \in \Xc$. It is \emph{weakly convex with modulus $L$} if
$g(x) - g(y) - \nabla g(y)^T (x - y) \ge -\frac{L}{2} \|x-y\|^2$ for all $x, y \in \Xc$.
An $L$-Lipschitz smooth function is $L$-weakly convex.

\section{Assumptions and Basic Properties}
\label{s:2}

We make the following boundedness, differentiability, and integrability assumptions.

\begin{description}
\item[(A1)] The set $\Bb$ is convex and compact.
\item[(A2)] The function $g(\Cdot)$ is continuously differentiable, its gradient is differentiable in the generalized sense, and constants $L_g$ and ${L_{\nabla g}}$ exist,
such that $\| \nabla g(u)\| \le L_g$ and $\| \nabla g(u) - \nabla g(v)\| \le {L_{\nabla g}}\|u -v\|$ for all $u,v \in \Rb^{n_f}$.
\item[(A3)] For every $x \in \Rb^{n_x}$ and every $y\in \Rb^{n_y}$ the function $f(x,y,\Cdot)$ is differentiable in the generalized sense
with the generalized subdifferential $\hat{\partial} f(x,y,\Cdot)$. Furthermore,  a function ${L_f}(x,y)$ and  constants $\widebar{L}_f>0$ and $C_f>0$  exist, such that
for all $\beta\in \Bb$ and all $(x,y)$
\begin{gather*}
 \sup_{f_\beta \in \hat{\partial} f(x,y,\beta)}\|f_\beta\| \le {L_f}(x,y),\\
\Eb [ {L_f}(X,Y)^4 ] \le \widebar{L}_f^4,\\
\intertext{and}
\Eb\big[\|f(X,Y,\beta)\|^4 \big] \le C_f^4.
\end{gather*}

\item[(A4)] For every $x \in \Rb^{n_x}$ the function $\varPsi(x,\Cdot)$ is differentiable in the generalized sense
with the generalized subdifferential $\hat{\partial} \varPsi(x,\Cdot)$. Furthermore, a function ${L_\psi}(x)$ and  constants
$\widebar{L}_\varPsi>0$ and  $C_{\varPsi}>0$ exist, such that for all $x$
\begin{gather*}
\sup_{\theta \in \Rb^{n_\theta}}\  \sup_{\varPsi_\theta \in \hat{\partial} \varPsi(x,\theta)}\|\varPsi_\theta\| \le {L_\psi}(x),\\
\Eb [ {L_\psi}(X)^4 ] \le \widebar{L}_\varPsi^4,\\
\intertext{and for some $\theta^0$,}
\Eb [ \|\varPsi(X,\theta^0)\|^4 ] \le C_\varPsi^4.
\end{gather*}
\end{description}

Assumption (A3) allows us to establish a fundamental property of generalized subdifferentials of conditional
expectations. We formulate it under general conditions which are implied by (A3).
\begin{lemma}
\label{l:conditional_subgradient}
Suppose the function $f(X,Y,\Cdot)$ is differentiable in the generalized sense and its subgradients are bounded
by an integrable function in any bounded neighborhood of any point $\beta$. Then for almost all $X$ the function
$\beta \mapsto \Eb\big[ f(X,Y,\beta) \big| X\big]$ is differentiable in the generalized sense, and
for every measurable function $s:\Rb^{n_x}\times \Rb^{n_y} \to \Rb^{n_f}$ such that $s(X,Y)  \in \hat{\partial}_\beta f(X,Y,\beta)$ a.s., we have
\[
\Eb\big[s(X,Y) \big| X\big] \in \hat{\partial} \Eb\big[ f(X,Y,\beta) \big| X\big] \quad  a.s.
\]
\end{lemma}

\begin{proof}
The probability measure $P$ on $\Rb^{n_x}\times \Rb^{n_y}$
can be disintegrated into a marginal measure $P_X$ on $\Rb^{n_x}$ and a kernel $T:\Rb^{n_x} \to \Pc(\Rb^{n_y})$,
 where $\Pc(\Rb^{n_y})$ is the space of probability measures
 on $\Rb^{n_y}$. Then
\[
\Eb\big[ f(X,Y,\beta) \big| X=x\big] = \int f(x,y,\beta)\;T(\D y|x).
\]
For every bounded neighborhood $K$ of $\beta$, the Lipschitz constant $L_K(x,y)$ of the function $f(x,y,\Cdot)$
is integrable, and thus, for almost
all $x$, its conditional expectation
$\int L_K(x,y)\,T(dy|x)$ exists as well.
The assertion then follows from Theorem \ref{t:expected-generalized}.
\end{proof}

We can compactly write the assertion of the above lemma as\footnote{Here and elsewhere, the subdifferentiation is with respect to the parameters $\beta$ or $\theta$; we omit the corresponding subscripts of $\hat{\partial}$ to alleviate the notation.}
\begin{equation}
\label{dFbeta}
\hat{\partial}_{  \beta} F(X,\beta) = \Eb \big[ \hat{\partial}_{  \beta} f(X,Y,\beta)\,\big|\, X \big] \quad \text{a.s.,}
\end{equation}
where the conditional expectation on the right-hand side is understood as the collection of conditional expected values of all measurable selections of the multifunction
$(X,Y) \mapsto \hat{\partial}_{  \beta} f(X,Y,\beta)$.

Furthermore, by Theorem \ref{t:composition} and Assumption (A2), the composition
$g( F(x,\Cdot))$ is differentiable in the generalized sense as well, and
\[
\hat{\partial}_{  \beta} [g\circ F](x,\beta)  = \big\{    F_\beta(x,\beta) \nabla g(F(x,\beta)) : F_\beta(x,\beta) \in \hat{\partial}_{  \beta} F (x,\beta)\big\}
\]
is its generalized subdifferential (we adopt the convention that the generalized Jacobian $F_\beta$ is an $n_\beta \times n_f$ matrix).
Thanks to the integrability condition in (A3) and Theorem \ref{t:expected-generalized}, the function $G(\Cdot)$ in \eqref{cso}
is differentiable in the generalized sense as well, and a realization of its generalized subdifferential is the multifunction
\begin{equation}
    \label{G-sub}
\begin{aligned}
\hat{\partial}_{  \beta} G(\beta) &= \Eb \big\{ \hat{\partial}_{  \beta} [g\circ F](X,\beta) \big\}
=\Eb \big\{
\hat{\partial}_{  \beta} F (X,\beta) \nabla g(F(X,\beta))
\big\}
\\ &= \Eb \big\{
\hat{\partial}f (X,Y,\beta) \nabla g(F(X,\beta))
\big\}.
\end{aligned}
\end{equation}
The last equation follows by conditioning.

Due to Assumptions (A3) and (A4), by virtue of Theorem \ref{t:composition},
the error function \eqref{te:Q} is differentiable in the generalized sense as well, and
the multifunction
\begin{equation}
\label{dQ}
\hat{\partial}_{(\beta,\theta)} Q(\beta,\theta) =
\Eb\left\{ \begin{bmatrix}
\hat{\partial}_{  \beta} f(X,Y,\beta) \\
-\hat{\partial}_{  \theta} \varPsi(X,\theta)
\end{bmatrix}
 \big(   F(X, \beta) - \varPsi(X, \theta)\big) \right\}
\end{equation}
is its generalized subdifferential.
Since the set \eqref{dQ} is a Cartesian product, its $\theta$-part,
\[
\hat{\partial}_{\theta} Q(\beta,\theta) =
- \Eb\left\{
\hat{\partial}_{  \theta} \varPsi(X,\theta)
 \big(   F(X, \beta) - \varPsi(X, \theta)\big) \right\},
\]
is a subdifferential of $Q(\beta,\theta)$ with respect to $\theta$.

As mentioned before, for every $\beta \in {\Bb}$, a parameter value $\bar{\theta}(\beta)\in \Rb^{n_\theta}$ exists,
such that $Q\big(\beta,\bar{\theta}(\beta)\big)=0$. In addition to that, we assume the following
growth condition.
\begin{description}
\item[(A5)] A constant $M>0$ exists such that for all $\beta\in {\Bb}$ and all $\theta \in \Rb^{n_\theta}$
\[
 Q(\beta,\theta)  \le M \big[\dist\big(0,\hat{\partial}_{\theta} Q(\beta,\theta)\big)\big]^2.
\]
\end{description}
This assumption may be viewed as a form of a uniform \textit{{\L}ojasiewicz condition} \cite{lojasiewicz1963propriete} for the function $Q(\beta,\Cdot)$, for all $\beta \in {\Bb}$
(for a thorough exposition, see \cite{bolte2010characterizations} and the references therein).

\begin{remark}
\label{r:growth_lin}
{\rm
If $\varPsi(\Cdot,\Cdot)$ is defined as the linear regression model \eqref{regression} and the model is sufficiently rich, then Assumption (A5) is satisfied, provided that the feature vectors are linearly independent.   The stochastic generalized subdifferential $\hat{\partial}_\theta \varPsi(X,\theta) = [\psi_1(X),\dots,\break\psi_{d_\theta}(X)]^T$ is a singleton. We denote it as ${\psi}(X)$.
We have $F(X,\beta) = {\psi}(X)^T \bar{\theta}$ for some $\bar{\theta}$ (which is assumed to exist), and $\varPsi(X,\theta) = {\psi}(X)^T {\theta}$. Therefore
\[
\hat{\partial}_{\theta} Q(\beta,\theta) = -\Eb \big[ {\psi}(X) \big( F(X,\beta) -  \varPsi(X,\theta) \big)   \big] = -\Eb\big[{\psi}(X) {\psi}(X)^T\big] (\bar{\theta} - \theta) .
\]
If multicollinearity does not occur in the feature matrix, then $\Eb\big[{\psi}(X) {\psi}(X)^T\big]$ is nonsingular and positive definite, and we can diagonalize it as $\Eb\big[{\psi}(X) {\psi}(X)^T\big] = U^T D U$. Here, $U$ is orthogonal and $D = \text{diag}(\lambda_1,\dots,\lambda_{n_\theta})$, where $\{\lambda_i\}_{1\le i \le n_\theta}$ are positive eigenvalues. Then,
\begin{align*}
Q(\beta,\theta) &= \frac{1}{2}\Eb\big[\big\|{\psi}(X)^T(\bar{\theta} - \theta)\big\|^2\big]
= \frac{1}{2} (\bar{\theta} - \theta)^T \Eb\big[{\psi}(X) {\psi}(X)^T\big] (\bar{\theta} - \theta) \\
&= \frac{1}{2} (\bar{\theta} - \theta)^T U^T D U (\bar{\theta} - \theta) \\
&\le \frac{1}{2 \cdot \min_{1\le i \le n_\theta} \lambda_i} (\bar{\theta} - \theta)^T U^T D U \cdot U^T D U (\bar{\theta} - \theta) \\
&= \frac{1}{2 \cdot \min_{1\le i \le n_\theta} \lambda_i} \big\|\hat{\partial}_{\theta} Q(\beta,\theta)\big)\big\|^2.
\end{align*}
Therefore, assumption (A5) is satisfied with $M = 1/\big(2 \cdot \min_{1\le i \le n_\theta} \lambda_i\big) > 0$.
\hfill
}
\end{remark}

{
\begin{remark}
\label{r:growth_nonlin}
{\rm
{ For a generalized differentiable $\varPsi(X,\Cdot)$, the  subgradients $Q_\theta(\beta, \theta) \in \hat{\partial}_{\theta} Q(\beta,\theta)$ can be derived with the use of Theorem \ref{t:chain}}
as follows:
\begin{equation*}
\begin{aligned}
Q_\theta(\beta, \theta) &= -\Eb \big[ \varPsi_\theta(X, \theta) \big( F(X,\beta) -  \varPsi(X,\theta) \big) \big]
= -\Eb \big[ \varPsi_\theta(X, \theta) \big( \varPsi(X,\bar{\theta}) -  \varPsi(X,\theta) \big) \big] \\
&= -\Eb \big[ \varPsi_\theta(X, \theta) \int_0^1 \varPsi_\theta(X, \tilde{\theta}_\tau)^T \big( \bar{\theta} - \theta \big) \;d\tau \big] \\
&= -\int_0^1 \Eb[\varPsi_\theta(X, \theta) \varPsi_\theta(X, \tilde{\theta}_\tau)^T]\; d\tau \cdot (\bar{\theta} - \theta),
\end{aligned}
\end{equation*}
where $\varPsi_\theta(X, \theta) \in \hat{\partial}_\theta \varPsi(X, \theta)$ and $\tilde{\theta}_\tau = \bar{\theta} + \tau (\theta - \bar{\theta})$.

Similar to the linear case, if $\Eb[\varPsi_\theta(X, \theta') \varPsi_\theta(X, \theta'')^T] \succeq \lambda_{\min} \cdot I$ for all $\theta', \theta'' \in  \Rb^{n_\theta}$, then

{
\[
\langle Q_\theta(\beta, \theta),\theta - \bar{\theta}\rangle \ge  \lambda_{\min} \| \theta - \bar{\theta} \|^2.
\]
Thus,
\[
\|Q_\theta(\beta, \theta)\|  = \max_{\|z\|=1} \, \langle Q_\theta(\beta, \theta), z \rangle \ge \lambda_{\min}\| \theta - \bar{\theta} \|.
\]
From Assumption (A4) we obtain
\[
 Q(\beta,\theta) =  \frac{1}{2}\Eb  \big[ \| \varPsi(X,\bar{\theta}) - \varPsi(X,\theta)\|^2 \big] \\
\le  \frac{1}{2} \widebar{L}_\varPsi^2 \| \bar{\theta} - \theta\|^2.
\]
The last two  inequalities imply that
Assumption (A5) is satisfied with $M = \widebar{L}_\varPsi^2/\big(2 \cdot \lambda_{\min}\big) > 0$.
}
}
\hfill
\end{remark}
}

Finally, we need an assumption about the behavior of the tracking error for very large $\theta$.

\begin{description}
\item[(A6)] A constant $\underline{R}>0$ exists such that if $\beta\in {\Bb}$ and  $\|\theta\|\ge \underline{R}$, then for all $Q_\theta \in\hat{\partial}_{\theta} Q(\beta,\theta)  $,
\[
\langle \theta, Q_\theta \rangle >0.
\]
\end{description}
{  This assumption is purely technical; its only purpose is to help keep
the set of points $\theta$ generated by our method bounded, without introducing
false stationary points on the boundary of the ball $\{\theta:  \|\theta\| \le \underline {R}\}$.}
\begin{remark}
\label{r:infty_lin}
{\rm
If $\varPsi(\Cdot,\Cdot)$ is defined as a linear regression model \eqref{regression} and the model is sufficiently rich, then Assumption (A6)
is satisfied, provided that the feature vectors are linearly independent.  As in Remark \ref{r:growth_lin},
\[
\hat{\partial}_{\theta} Q(\beta,\theta) = -\Eb \big[ {\psi}(X) \big( F(X,\beta) - {\psi}(X)^T {\theta}\big)\big].
\]
Therefore, using the Cauchy--Schwartz inequality, we obtain
\begin{align*}
\langle \theta, Q_\theta \rangle  &= \theta^T \Eb\big[{\psi}(X) {\psi}(X)^T\big] \theta
- \theta^T \Eb \big[ {\psi}(X)  F(X,\beta)\big]\\
&\ge \theta^T \Eb\big[{\psi}(X) {\psi}(X)^T\big] \theta
- \|\theta\| \Big( \Eb \big[{\psi}(X)^T {\psi}(X)\big]
\Eb\big[\|F(X,\beta)\|^2 \big]\Big)^{1/2}\\
& \ge \|\theta\|^2 \min_{1\le i \le n_\theta} (\lambda_i)  - \|\theta\| \Big( \max_{1\le i \le n_\theta} (\lambda_i)
\, \max_{\beta\in {\Bb}} \Eb\big[\|F(X,\beta)\|^2 \big]\Big)^{1/2}.
\end{align*}
Due to Assumption (A3), a constant $C$ exists, such that the term after the minus sign above is bounded by $C\|\theta\|$, which implies that $\langle \theta, Q_\theta \rangle \ge 0$ for all sufficiently
large $\|\theta\|$.
\hfill
}
\end{remark}

{
\begin{remark}
\label{r:infty_nonlin}
{\rm
For a generalized differentiable $\varPsi(X,\Cdot)$,
applying the same technique as in Remark \ref{r:growth_nonlin}, we obtain
\begin{equation*}
\begin{aligned}
Q_\theta(\beta, \theta) &= -\Eb \big[ \varPsi_\theta(X, \theta) \big( F(X,\beta) - \varPsi(X,0) -\varPsi(X,\theta)+\varPsi(X,0) \big) \big]\\
&= -\Eb \big[ \varPsi_\theta(X, \theta) \big( F(X,\beta) - \varPsi(X,0)\big]
+\int_0^1 \Eb[\varPsi_\theta(X, \theta) \varPsi_\theta(X, \tau\theta)^T]\; d\tau \cdot \theta.
\end{aligned}
\end{equation*}
If for all $\theta', \theta'' \in \Rb^{n_\theta}$, we have $\Eb[\varPsi_\theta(X, \theta') \varPsi_\theta(X, \theta'')^T] \succeq \lambda_{\min}   I$, then
\[
\langle Q_\theta(\beta, \theta),\theta \rangle \ge \lambda_{\min} \|\theta\|^2
-\widebar{L}_\varPsi  \max_{\beta\in \Bb}\Big(\Eb\big[\|F(X,\beta) - \varPsi(X,0)\|^2\big]\Big)^{1/2} \|\theta\|\ge 0,
\]
for all sufficiently large $\|\theta\|$.}
\hfill
\end{remark}
}

\section{The method}
\label{s:3}
The necessary optimality condition  for problem \eqref{cso} has the form
\begin{equation}
\label{optimality}
0 \in \hat{\partial}_{ \beta} G(\beta) + \Nc_{\Bb}(\beta),
\end{equation}
where $\hat{\partial}_{ \beta} G(\Cdot)$ is the generalized subdifferential \eqref{G-sub} and $\Nc_{\Bb}(\Cdot)$ is the normal cone
to the set $\Bb$. We define the sets
\begin{align*}
    {\Bb}^* &= \big\{ \beta\in {\Bb} : \ \text{\eqref{optimality} is satisfied} \big\},\\
    \Zb^* &= \big\{ (\beta,\theta)\in {\Bb}^* \times \Rb^{n_\theta}:  Q(\beta,\theta) = 0\big\}.
\end{align*}
We propose a stochastic subgradient method for finding an element of $\Zb^*$.


To simplify the development and analysis
of the method, we assume that $R \ge \underline{R}$ is known,
so that $\Zb^* \subset \Bb \times \Theta_R$,
where $\Theta_R = \{ \theta\in \Rb^{n_\theta}: \|\theta \| \le R\}$.
Due to Assumption (A6), any $R\ge \underline{R}$ has this property.

At each iteration $k=0,1,\dots$ of the method,
given the current approximate solution $(\beta^k,\theta^k)\in {\Bb} \times \Theta_R$,
we use an observation $(X^k,Y^k)$ from the joint distribution of $(X,Y)$, to construct random directions
$\tilde{d}_\beta^k$ and $\tilde{d}_\theta^k$, and to update the current point:
\begin{equation}
\label{method}
\begin{aligned}
\beta^{k+1} &= \varPi_{\Bb}(\beta^k + \tau_k \tilde{d}_\beta^k),\\
\theta^{k+1} &= \varPi_{\Theta_R}(\theta^k + \tau_k \tilde{d}_\theta^k).
\end{aligned}
\end{equation}
Here, $\tau_k>0$ is a stepsize of the method, $\varPi_{\Bb}(\Cdot)$  denotes the orthogonal projection on $\Bb$,
and $\varPi_{\Theta_R}(\Cdot)$ is the orthogonal projection on the ball $\Theta_R$.

To construct the directions $\tilde{d}_\beta^k$ and $\tilde{d}_\theta^k$,
we use the values and  generalized gradients of the functions $f(X^k,Y^k,\Cdot)$ and $\varPsi(X^k,\Cdot)$, which
can be observed at the current point $(\beta^k,\theta^k)$. We also use the derivatives of the function $g(\Cdot)$ at the
points $\varPsi(X^k,\theta^k)$, which can be observed as well.

In our presentation of the method and its analysis, we denote by $(\varOmega,\Fc,P)$ the algorithmic probability space on which
the random sequences generated by the method are defined, and by
$\Fc_k$ the $\sigma$-subalgebra defined by the history $(\beta^0,\theta^0,X^0,Y^0,\dots,X^{k-1},Y^{k-1},\beta^k,\theta^k)$.
We assume that at iteration $k=0,1,\dots$ we observe the pair $(X^k,Y^k)$, independently of $\Fc_k$, and we calculate the following quantitites:
\begin{gather*}
f^k  = f(X^k,Y^k,\beta^k),\\
f_\beta^k  \in \hat{\partial}_{ \beta} f(X^k,Y^k,\beta^k),\\
\varPsi^k  = \varPsi(X^k,\theta^k),\\
\varPsi_\theta^k  \in \hat{\partial}_{ \theta} \varPsi(X^k,\theta^k),\\
\nabla g^k = \nabla g(\varPsi^k).
\end{gather*}
Then we calculate the directions in \eqref{method}:
\begin{subequations}
\label{dzk}
\begin{align}
    \tilde{d}_\beta^k &= - f_\beta^k \,\nabla g^k,  \label{dbetak}\\
    \tilde{d}_\theta^k &=\gamma \,\varPsi_\theta^k (f^k - \varPsi^k). \label{dthetak}
\end{align}
\end{subequations}
Here, $\gamma>0$ is a parameter of the method.
In the next section, we shall demonstrate that
$\tilde{d}_\theta^k$ is a negative stochastic subgradient with respect to $\theta$ of the error function \eqref{te:Q} at $(\beta^k,\theta^k)$. However, the
direction $\tilde{d}_\beta^k$ is not a negative stochastic subgradient of any function, because we use the observed gradient $\nabla g(\varPsi^k)$ instead of
the inaccessible $\nabla g\big(F(X^k,\beta^k)\big)$. Thus, the
pair $(\tilde{d}_\beta^k,\tilde{d}_\theta^k)$ is \textit{not
a negative stochastic subgradient of any function}.
{  The method \eqref{method}--\eqref{dzk} defines a complex dynamics in the
space $\Rb^{n_\beta}\times \Rb^{n_\theta}$; we will analyze it with the use of a specially
designed Lyapunov function.}

We make the following standard assumptions about the stepsizes in the method \eqref{method}.

\begin{description}
\item[(A7)]
The stepsizes $\tau_k$ are $\Fc_k$-measurable, and such that\vspace{0.5ex}
\begin{tightlist}{iii}
\item  $\tau_k >0, \quad k=0,1,2,\dots$;
\item   $ \displaystyle{\sum_{k=0}^\infty} \tau_k = \infty \quad \text{a.s.}$;
\item $\Eb \displaystyle{\sum_{k=0}^\infty} \tau_k^2 < \infty$.
\end{tightlist}
\end{description}

\section{Convergence Analysis}
\label{s:4}
Consider the multifunction
$\varGamma: \Bb \times \Theta_R \rightrightarrows \Rb^{n_f}$ defined as follows:
\[
\varGamma(\beta,\theta) = \Eb
\left\{
\begin{bmatrix}
 -\hat{\partial}_\beta F(X,\beta) \nabla g\big(\varPsi(X,\theta)\big) \\
\gamma \,\hat{\partial}_\theta \varPsi(X,\theta)\big(F(X,\beta) -\varPsi(X,\theta)\big)
\end{bmatrix}
\right\}.
\]
We need to establish the properties of the multifunction $\varGamma(\Cdot,\Cdot)$ that are relevant for the convergence analysis.
\begin{lemma}
\label{l:Gamma}
The multifunction $\varGamma(\Cdot,\Cdot)$ is nonempty-, convex-, and compact-valued and upper-semicontinuous.
\end{lemma}
\begin{proof}
By virtue of Lemma \ref{l:conditional_subgradient}, the $\beta$-component of $\varGamma(\beta,\theta)$ equals
\begin{multline*}
 -\Eb \big[ \hat{\partial}_\beta F(X,\beta) \nabla g\big(\varPsi(X,\theta)\big)\big] =
 -\Eb\Big\{ \Eb\big[ \hat{\partial}_\beta f(X,Y,\beta)\big|X\big] \nabla g\big(\varPsi(X,\theta)\big)\Big\}\\
  =
 - \Eb \big[ \hat{\partial}_\beta  f(X,Y,\beta) \nabla g\big(\varPsi(X,\theta)\big)\big].
\end{multline*}

The $\theta$-component of $\varGamma(\beta,\theta)$ can be written as follows:
\[
\gamma \,\Eb \big[ \hat{\partial}_\theta \varPsi(X,\theta)\big(F(X,\beta) -\varPsi(X,\theta)\big)\big]
= \gamma \,\Eb \big[ \hat{\partial}_\theta \varPsi(X,\theta)\big(f(X,Y,\beta) -\varPsi(X,\theta)\big)\big].
\]
Due to Assumptions (A2)--(A4), for all $(X,Y)$, the multifunction
\[
(\beta,\theta) \mapsto \widetilde{\varGamma}(\beta,\theta;X,Y) \triangleq
\begin{bmatrix}
 -\hat{\partial}_{ \beta} f(X,Y,\beta) \nabla g\big(\varPsi(X,\theta)\big) \\
\gamma \,\hat{\partial}_{ \theta} \varPsi(X,\theta)\big(f(X,Y,\beta) -\varPsi(X,\theta)\big)
\end{bmatrix}
\]
is upper-semicontinuous, and convex- and compact-valued. Furthermore, its $\beta$-values are
bounded by and integrable function $L_f(X,Y)L_g$, and its $\theta$-values are bounded
by an integrable function $  \gamma L_{\varPsi}(X)( C_f +  C_{\varPsi} + 4  L_{\varPsi}(X) R^2)$.
These two conditions imply that its integral -- the multifunction $\varGamma(\Cdot,\Cdot)$ -- is upper-semicontinuous, and convex- and compact-valued as well. This implication is part of
the proof of Theorem \ref{t:expected-generalized}; see \cite[Lem. 21.7]{mikhalevich1987nonconvex}.
\end{proof}

Define the average directions,
\begin{equation}
\label{def:avg_dir}
d^k = \Eb [\tilde{d}^k \,| \, \Fc_k],\quad k=0,1,\dots,
\end{equation}
and let
\begin{equation}
\notag
e^k = \tilde{d^k} - d^k.
\end{equation}

Lemma \ref{l:conditional_subgradient} allows us to establish a key property of the conditional expectations of the directions used in the method.

\begin{lemma}
\label{l:dk-in-Gamma}
For all $k=0,1,2,\dots$ we have $d^k \in \varGamma(\beta^k,\theta^k)$.
\end{lemma}
\begin{proof}
Since the observation $(X^k,Y^k)$ is independent of $\Fc_k$, Lemma \ref{l:conditional_subgradient}  yields the relations
\begin{align*}
\Eb[\tilde{d}_\beta^k|\Fc_k]  &\in -\Eb \big[ \hat{\partial}_{ \beta} f(X^k,Y^k,\beta^k) \nabla g(\varPsi(X^k,\theta^k)) ~\big|~ \Fc_k\big] \\
&= -\Eb \Big\{ \Eb \big[ \hat{\partial}_{ \beta} f(X^k,Y^k,\beta^k)  ~\big|~ \beta^k,\theta^k, X^k \big] \nabla g(\varPsi(X^k,\theta^k)) ~\Big|~ \beta^k,\theta^k \Big\}\\
&= -\Eb \big[ \hat{\partial}_{ \beta} F(X^k,\beta^k) \nabla g(\varPsi(X^k,\theta^k)) ~\big|~ \beta^k,\theta^k \big]
= \varGamma_\beta (\beta^k,\theta^k).
\end{align*}
The $\theta$-component of the direction can also be treated by conditioning:
\begin{align*}
\Eb[\tilde{d}_\theta^k|\Fc_k] &\in \gamma \Eb \big[ \hat{\partial}_{ \theta} \varPsi(X^k,\theta^k)\big(f(X^k,Y^k,\beta^k) -\varPsi(X^k,\theta^k)\big)\big| \Fc_k\big]\\
&= \gamma \Eb \Big\{\hat{\partial}_{ \theta} \varPsi(X^k,\theta^k)\Eb\big[f(X^k,Y^k,\beta^k) -\varPsi(X^k,\theta^k)\big| X^k,\beta^k,\theta^k\big]\Big| \beta^k,\theta^k\Big\}\\
&= \gamma \Eb \big\{\hat{\partial}_{ \theta} \varPsi(X^k,\theta^k)\big(F(X^k,\beta^k) -\varPsi(X^k,\theta^k)\big)\big| \beta^k,\theta^k\big\} = \varGamma_\theta(\beta^k,\theta^k).
\end{align*}
\end{proof}

We can also estimate the conditional variance of the errors $\{e^k\}_{k\ge 0}$.
\begin{lemma}
    \label{l:e2-bounded}
Under Assumptions \textup{(A1)--(A4)}, the directions $\{d^k\}_{k\ge 0}$ are well-defined and a constant $\sigma > 0$ exists, such that
$\Eb\big[\|e^k\|^2 \,\big|\,\Fc_k\big] \le \sigma^2$, for all $k \ge 0$.
\end{lemma}
\begin{proof}
For the $\beta$-component $e_\beta^k$ of $e^k$,  we have
\[
\tilde{d}_\beta^k = -f_\beta(X^k,Y^k,\beta^k)\nabla g(\varPsi(X^k,\theta^k)),
\]
and thus, by Assumptions (A2) and (A3),
\[
\Eb\big[ \|{e}_\beta^k\|^2\big| \Fc_k ] \le \Eb\big[ \|\tilde{d}_\beta^k\|^2\big| \Fc_k ]
\le L_g^2 \Eb \big[ \| f_\beta(X^k,Y^k,\beta^k)\|^2 \big| \Fc_k \big]\le
 L_g^2 \widebar{L}_f^2.
\]
For the $\theta$-component,
\begin{align*}
\Eb\big[ \| \tilde{d}_\theta^k\|^2\big| \Fc_k\big]
&\le \gamma^2 \Eb\big[ \|\varPsi_\theta(X^k,\theta^k)\|^2 \cdot \| f(X^k,Y^k,\beta^k)  - \varPsi(X^k,\theta^k)\|^2\big| \Fc_k\big]\\
&\le \gamma^2 \Eb\big[{L_\psi}(X^k)^2\| f(X^k,Y^k,\beta^k)  - \varPsi(X^k,\theta^k)\|^2\big| \Fc_k\big]\\
&\le 2\gamma^2 \Eb\big[{L_\psi}(X^k)^2\| f(X^k,Y^k,\beta^k)\|^2\big| \Fc_k\big]\\
&{\quad} +2 \gamma^2 \Eb\big[ {L_\psi}(X^k)^2\|\varPsi(X^k,\theta^k)\|^2\big| \Fc_k\big].
\end{align*}
Using the Cauchy--Schwartz inequality and Assumptions (A3) and (A4) we obtain
\[
\Eb\big[{L_\psi}(X^k)^2\| f(X^k,Y^k,\beta^k)\|^2\big| \Fc_k\big]
\le  \widebar{L}_\varPsi^2 C_f^2.
\]
In a similar way, by (A4),
\begin{align*}
\lefteqn{\Eb\big[ {L_\psi}(X^k)^2\|\varPsi(X^k,\theta^k)\|^2\big| \Fc_k\big]}\\
&\le 2 \Eb\big[{L_\psi}(X^k)^2\big( \|\varPsi(X^k,\theta^0)\|^2 + {L_\psi}(X^k)^2 \|\theta^k-\theta^0\|^2\big)\big| \Fc_k\big]
\le 2 \widebar{L}_\varPsi^2 C_\varPsi^2 + 8R^2 \widebar{L}_\varPsi^4.
\end{align*}
Combining the last three inequalities, we conclude that
$\Eb\big[ \| e_\theta^k\|^2\big| \Fc_k\big] $, $k\ge 0$, are uniformly bounded as well.
\end{proof}
\begin{remark}
It is evident from the second part of the proof that if $L_\varPsi(X)$ is bounded by a constant, then
we only need
bounds on the second moments in \textup{(A3)} and \textup{(A4)}, rather than the fourth moments. Lemma \textup{\ref{l:e2-bounded}}
is the only result using the assumptions about the fourth moments.
\end{remark}
We can now analyze the convergence of the method \eqref{method}
by the differential inclusion approach {  originating from
\cite{volkonskii1970convergence}, and further developed in
\cite{belen1974iterative,benaim2005stochastic,benaim2006stochastic}};
see also \cite{ljung1977analysis,borkar2009stochastic,kushner2012stochastic,harold1997stochastic}.

For brevity, denote $z^k = (\beta^k,\theta^k) $, $k=0,1,\dots$, and $\Zb=\Bb \times\Theta_R$.
We consider a particular trajectory $z^k(\omega) = (\beta^k(\omega),\theta^k(\omega)) $, $k=0,1,\dots$; in our notation, we omit
the elementary event $\omega\in \Omega$, which may be arbitrary, except for a zero measure set.

First, we construct continuous-time trajectories by linear interpolation. We introduce the accumulated stepsizes
\[
t_k = \sum_{j=0}^{k-1}\tau_j,\quad k=0,1,\dots,
\]
and we construct the interpolated trajectory
\[
Z_0(t) = z^k + \frac{t-t_{k}}{\tau_k}(z^{k+1}-z^k),\quad t_{k}\le t \le t_{k+1},\quad k=0,1,\dots.
\]
Then, for an increasing sequence of positive numbers $\{s_m\}$ diverging to $\infty$, we define the shifted trajectories
\begin{equation}
\label{Z-shifted}
Z_m(t)=Z_0(t+s_m), \quad m=0,1,\dots,\quad t \ge 0.
\end{equation}
In particular, if we set $s_m = t_m$, we obtain $Z_m(0) = z^m$.

The following theorem follows directly from \cite[Thm. 4.1]{majewski2018analysis}.
\begin{theorem}
\label{t:diff-inc-beta}
Under Assumptions \textup{(A1)--(A4)} and \textup{(A7)},
the sequence $\big\{Z_m(\Cdot)\big\}_{m\ge 0}$ has a subsequence convergent a.s. to an absolutely continuous function $Z_\infty:[0,\infty) \to \Xc$ which is a solution of the differential inclusion:
\begin{equation}
\label{diff-inc-beta}
\dt{Z}_\infty(t) \in  \varGamma(Z_{\infty}(t))-\Nc_{\Zb}(Z_\infty(t)), \quad t \ge 0.
\end{equation}
Moreover, for any $t\ge 0$, the point $Z_\infty(t)$
is an accumulation point of the sequence $\{z^k\}_{k\ge 0}$.
\end{theorem}
Since any solution of \eqref{diff-inc-beta} must lie in $\Zb$,
we can  write the differential inclusion \eqref{diff-inc-beta} in an equivalent way \cite{aubin2012differential}:
\begin{equation}
\label{diff-inc-beta-T}
\dt{Z}_\infty(t) \in \varPi_{\Tc_{\Zb}(Z_\infty(t))}\big( \varGamma(Z_\infty(t))\big), \quad t \ge 0,
\end{equation}
which underlines the nature of the limiting subgradient process. Here, $\Tc_\Zb(z)$ is the tangent cone to the convex set $\Zb$ at $z$.

Our intention is to show that all equilibria of the system \eqref{diff-inc-beta} lie in $\Zb^*$.

To this end, we construct a specially tailored Lyapunov function with parameters $\alpha > 0$ and $\lambda \ge 0$:
\begin{equation}
\label{Lyapunov}
    w(\beta,\theta) \triangleq G(\beta) + \alpha \Delta^\lambda(\beta, \theta),
\end{equation}
where $G(\beta)$ is the CSO loss function in \eqref{cso} and
\begin{equation}
\Delta^{\lambda}(\beta,\theta) \triangleq \Eb\Big[g(F) - g(\varPsi) - \langle \nabla g(\varPsi), F - \varPsi \rangle + \frac{\lambda}{2} \|F-\varPsi\|^2\Big].
\label{te:g}
\end{equation}
If $\lambda > {L_{\nabla g}}$, where ${L_{\nabla g}}$ is defined in Assumption (A2), then
$\Delta^\lambda(\beta, \theta) \ge 0$ with strict inequality unless \eqref{tracking-exact} holds true.
Therefore, we can use $\Delta^\lambda(\beta, \theta)$  as
a tracking error of our parametric model. The values of $\alpha$, $\lambda$, and $\gamma$ will be chosen later.

 As mentioned in the Introduction,
$\Delta^0(\beta, \theta)$ is a  stochastic generalization of the Bregman distance; the function \eqref{te:g} is its regularized version.

\begin{lemma}
\label{l:w-gendif}
Under Assumptions \textup{(A2)--(A4)} the Lyapunov function $w(\Cdot,\Cdot)$ is differentiable in the generalized sense.
\end{lemma}
\begin{proof}
In \S \ref{s:2} we established the generalized differentiability of the functions $G(\Cdot)$ and $Q(\Cdot,\Cdot)$.  It remains to consider the function
\[
\Delta^{0}(\beta,\theta) = \Eb\Big[g(F) - g(\varPsi) - \langle \nabla g(\varPsi), F - \varPsi \rangle \Big],
\]
and the expression under the expected value:
\begin{equation}
    \label{delta0-beta}
\delta^{0}(\beta,\theta) = g(F) - g(\varPsi) - \langle \nabla g(\varPsi), F - \varPsi \rangle.
\end{equation}
By Assumptions (A2)--(A4) and Theorem \ref{t:composition},
$\delta^{0}(\Cdot,\Cdot)$ is Norkin differentiable. Its generalized subgradients with respect to $\beta$ have the form
\[
\delta^{0}_\beta = F_\beta\big( \nabla g(F) - \nabla g(\Psi)\big),
\]
and are bounded by $2L_g\Eb[\widebar{L}_f(X,Y)|X]$. It is an integrable function,
thanks to (A2).
In a similar way, by Assumptions (A2) and (A4) and Theorem \ref{t:composition}, the generalized subgradients
with respect to $\theta$ have the form
\begin{equation}
    \label{delta0-theta}
\delta^{0}_\theta = \varPsi_\theta H(\varPsi) ( F - \Psi),
\end{equation}
 or are convex combinations of elements of this structure,
 where we denote by $H(\varPsi)$
an arbitrary element of $\hat{\partial}(\nabla g)(\varPsi)$. Thus, their norm is
bounded by $L_{\nabla g}L_\varPsi(X)\|F-\varPsi\|$. As $L_\varPsi$
and $\|F-\varPsi\|$ are square-integrable, the bound is integrable, by virtue
of the Cauchy--Schwartz inequality.  Applying Theorem \ref{t:expected-generalized},
we conclude that the expected value function $\Delta^0$ is Norkin differentiable as well.
\end{proof}

Following the differential inclusion method, consider the function $W(t) = w(\beta(t),\theta(t))$, $t \ge 0$, where $(\beta(\Cdot),\theta(\Cdot)) \triangleq Z(\Cdot)$ is a solution of the
inclusion \eqref{diff-inc-beta}.

As the Lyapunov function $w(\Cdot,\Cdot)$ is Norkin differentiable, by virtue of Theorem \ref{t:composition}, its
weak derivative can be expressed by the chain rule:
\begin{equation}
\dt{W}
= \langle w_\beta, \dt{\beta} \rangle + \langle w_\theta, \dt{\theta} \rangle
= \langle G_\beta + \alpha \Delta^\lambda_\beta, \dt{\beta} \rangle + \alpha \langle \Delta^\lambda_\theta, \dt{\theta} \rangle.
\label{grad:P}
\end{equation}
Due to Theorem \ref{t:chain}, the integral  $\int_0^T \dt{W}(t)\,\D t$ on any interval of time does not depend on the particular selections of the generalized subgradients used in this calculation.

In the following lemmas, we estimate the norms of the generalized subgradients of some relevant functions involved in the estimation of the derivative of the Lyapunov function \eqref{Lyapunov}.
\begin{lemma}
Under Assumptions \textup{(A2)} and \textup{(A4)},  $\big|\langle \Delta^0_\theta, Q_\theta \rangle \big| \le 2 {L_{\nabla g}} \widebar{L}_\varPsi^2 Q$.
\label{l:QG_delta_theta}
\end{lemma}
\begin{proof}
In view of \eqref{delta0-theta},  $\Delta^0_\theta = \Eb\big[ \varPsi_\theta H(\varPsi)(F-\varPsi)\big]$. By virtue of Assumption (A2), $\|H(\varPsi)\| \le L_{\nabla g}$.
Also,
$Q_\theta = \Eb\big[ \varPsi_\theta(F-\varPsi)\big]$.
Using the Cauchy--Schwartz inequality,
we obtain the following chain of estimates:
\begin{align*}
\big|\langle \Delta^0_\theta, Q_\theta \rangle\big|  &\le \| \Delta^0_\theta\| \| Q_\theta \|
\le {L_{\nabla g}} \Eb[\|\varPsi_\theta\|^2] \Eb[\|F - \varPsi\|^2] \le 2 {L_{\nabla g}} \widebar{L}_\varPsi^2 Q.
\end{align*}
The last inequality uses \eqref{te:Q}.
\end{proof}

\begin{lemma}
Under Assumption \textup{(A3)},  $\|Q_\beta\|^2 \le 2 \widebar{L}_f^2 Q$.
\label{l:QG_Q_beta}
\end{lemma}
\begin{proof}
By the  Cauchy--Schwartz inequality,

\[
\|Q_\beta\| = \big\|\Eb[ F_\beta(F-\varPsi)]\big\| \le \sqrt{\Eb[\|F_\beta\|^2]}\sqrt{ \Eb[\|F-\varPsi\|^2]} \le  \widebar{L}_f \sqrt{2 Q}.
\]
\end{proof}
\begin{lemma}
Under Assumptions \textup{(A2)} and \textup{(A3)}, $\|\Delta^0_\beta\|^2 \le 2 L_{\nabla g}^2 \widebar{L}_f^2 Q$.
\label{l:QG_delta_beta}
\end{lemma}
\begin{proof}
By \eqref{delta0-beta}, $\Delta^0_\beta = \Eb[F_\beta(\nabla g(F) - \nabla g(\varPsi))$.
Similar to the proof of Lemma \ref{l:QG_Q_beta}, by the Cauchy--Schwartz inequality and Assumption (A2),
\begin{align*}
\|\Delta^0_\beta\|^2 &= \big\|\Eb[F_\beta(\nabla g(F) - \nabla g(\varPsi)) ]\big\|^2
\le \Eb[\|F_\beta\|^2]\, \Eb[\|\nabla g(F) - \nabla g(\varPsi)\|^2]  \\
&\le L_{\nabla g}^2 \Eb[\|F_\beta\|^2]\, \Eb[\|F - \varPsi\|^2]
\le 2L_{\nabla g}^2 \widebar{L}_f^2 Q.
\end{align*}
The last inequality uses Assumption (A3), as in Lemma \ref{l:QG_Q_beta}.
\end{proof}


{
We can now estimate the rate of descent of the Lyapunov function \eqref{Lyapunov}.
\begin{proposition}
Under assumptions \textup{(A2)} and \textup{(A3)},
\[
\langle w_\beta, \dt{\beta} \rangle \le \frac{\widebar{L}_f^2}{2} \big( (\alpha+1){L_{\nabla g}} + \alpha \lambda \big)^2 Q.
\]
Furthermore, if $Q=0$, then
\[
\langle w_\beta, \dt{\beta} \rangle = - \big\| \varPi_{\Tc_{\Bb}(\beta)}\big(-G_\beta \big)\big\|^2.
\]
\label{p:QG_L_delta_Q_beta}
\end{proposition}
\begin{proof}
The first coordinate of the differential inclusion \eqref{diff-inc-beta} reads:
\begin{equation}
\dt{\beta} \in \varPi_{\Tc_{\Bb}(\beta)}\big(-\Eb[F_\beta \nabla g(\varPsi) ] \big) = \varPi_{\Tc_{\Bb}(\beta)}\big( -G_\beta + \Delta^{0}_\beta \big),
\label{grad:beta}
\end{equation}
For any solution $(\beta(t), \theta(t))$ of the inclusion \eqref{t:diff-inc-beta}, with the corresponding element $N(t) \in \Nc_{\Bb}(\beta(t))$, we define the projection matrix
\[
P(t)  = I - \frac{N(t) N(t)^T}{\|N(t)\|^2}.
\]
Then
\begin{equation}
\dt{\beta} = P (-G_\beta + \Delta^{0}_\beta).
\label{grad:beta_proj}
\end{equation}
Using the facts that $P$ is nonexpansive, $P^T=P$, and $P^2=P$,
by elementary manipulations, we obtain an upper bound:
\begin{equation}
\label{w-dot-chain}
\begin{aligned}
\lefteqn{\langle w_\beta, \dt{\beta} \rangle = -\langle G_\beta + \alpha \Delta^{\lambda}_\beta, P(G_\beta - \Delta^{0}_\beta) \rangle} \\
&= -\Big\|P\big(G_\beta + \frac{\alpha}{2} \Delta^{\lambda}_\beta - \frac{1}{2} \Delta^0_\beta\big)\Big\|^2 + \Big\|P\big(\frac{\alpha}{2} \Delta^{\lambda}_\beta + \frac{1}{2} \Delta^0_\beta\big)\Big\|^2 \\
&\le \Big\|P\big(\frac{\alpha}{2} \Delta^{\lambda}_\beta + \frac{1}{2} \Delta^0_\beta\big) \Big\|^2 =
\Big\|P\big(\frac{\alpha+1}{2} \Delta^0_\beta + \frac{\alpha \lambda}{2} Q_\beta\big)\Big\|^2 \\
&= \Big\|\frac{\alpha+1}{2} P \Delta^0_\beta\Big\|^2 + \Big\|\frac{\alpha \lambda}{2} P Q_\beta\Big\|^2 + 2 \Big\langle \rho \frac{\alpha+1}{2} P \Delta^0_\beta, \varrho^{-1} \frac{\alpha \lambda}{2} P Q_\beta \Big\rangle\\
&\le \frac{(\alpha+1)^2}{4} (1+\rho^2)\|P \Delta^0_\beta\|^2 + \frac{\alpha^2 \lambda^2}{4}(1+\rho^{-2})\|P Q_\beta\|^2,
\end{aligned}
\end{equation}
for any $\rho>0$.
Lemmas \ref{l:QG_Q_beta} and \ref{l:QG_delta_beta} allow us to continue the last estimate as follows:
\begin{align*}
\langle w_\beta, \dt{\beta} \rangle &\le \min_{\rho > 0} \Big[ \frac{(\alpha+1)^2}{4} (1+\rho^2)\|P \Delta^0_\beta\|^2 + \frac{\alpha^2 \lambda^2}{4}(1+\rho^{-2})\|P Q_\beta\|^2 \Big] \\
&\le \min_{\rho > 0} \Big[ \frac{(\alpha+1)^2}{4} (1+\rho^2)\|\Delta^0_\beta\|^2 + \frac{\alpha^2 \lambda^2}{4}(1+\rho^{-2})\|Q_\beta\|^2 \Big] \\
&\le \min_{\rho > 0} \Big[ \frac{(\alpha+1)^2}{2} (1+\rho^2) L_{\nabla g}^2 + \frac{\alpha^2 \lambda^2}{2}(1+\rho^{-2}) \Big]   \widebar{L}_f^2 Q \\
&= \frac{\widebar{L}_f^2}{2} \big( (\alpha+1){L_{\nabla g}} + \alpha \lambda \big)^2 Q.
\end{align*}
This proves the first assertion.
Now, suppose $Q=0$. This implies that $\varDelta^\lambda=0$ for all $\lambda \ge 0$.
The first line of \eqref{w-dot-chain} yields
\[
\langle w_\beta, \dt{\beta} \rangle = - \| P G_\beta \|^2 = - \big\| \varPi_{\Tc_{\Bb}(\beta)}\big(-G_\beta \big)\big\|^2,
\]
as claimed.
\end{proof}

}


The second coordinate of the differential inclusion \eqref{diff-inc-beta} reads:
\begin{equation}
\dt{\theta} =  \varPi_{\Tc_{\Theta_R}(\theta)}\big(\gamma\, \Eb[(F - \varPsi)^T \varPsi_\theta] \big)
= \varPi_{\Tc_{\Theta_R}(\theta)}\big( -\gamma Q_\theta\big) = -\gamma Q_\theta,
\label{grad:simp}
\end{equation}
because, by virtue of  Assumption (A6), when $\|\theta\|=R$, then $ -\gamma Q_\theta \in \Tc_{\Theta_R}(\theta)$.

\begin{theorem}
\label{t:negative-der}
Under Assumptions \textup{(A1)--(A6)}, if
\begin{equation}
\label{negative-der}
\lambda > 2 \widebar{L}_\varPsi^2 M {L_{\nabla g}} \quad \text{and} \quad
\gamma > \frac{\widebar{L}_f^2 M\big( (\alpha + 1){L_{\nabla g}} + \alpha \lambda \big)^2}{2 \alpha \big(\lambda - 2 \widebar{L}_\varPsi^2 M {L_{\nabla g}} \big)} ,
\end{equation}
then $\dt{W} \le 0$. Furthermore, $\dt{W} < 0$ if $(\beta,\theta)\notin \Zb^*$.
\end{theorem}
\begin{proof}
Following \eqref{grad:P}, \eqref{grad:simp}, Lemma \ref{l:QG_delta_theta}, and Proposition \ref{p:QG_L_delta_Q_beta},
\begin{align*}
\dt{W}
&\le \frac{\widebar{L}_f^2}{2} \big( (\alpha+1){L_{\nabla g}} + \alpha \lambda \big)^2 Q - \alpha \gamma \langle \Delta^{\lambda}_\theta, Q_\theta \rangle \\
&= \frac{\widebar{L}_f^2}{2} \big( (\alpha+1){L_{\nabla g}} + \alpha \lambda \big)^2 Q - \alpha \gamma \big[\langle \Delta^0_\theta, Q_\theta \rangle + \lambda \|Q_\theta\|^2 \big] \\
&\le \frac{\widebar{L}_f^2}{2} \big( (\alpha+1){L_{\nabla g}} + \alpha \lambda \big)^2 Q + \alpha \gamma |\langle \Delta^0_\theta, Q_\theta \rangle| - \alpha \gamma \lambda \|Q_\theta\|^2 \\
&\le \Big[\frac{\widebar{L}_f^2}{2} \big( (\alpha+1){L_{\nabla g}} + \alpha \lambda \big)^2 + \alpha \gamma \cdot 2{L_{\nabla g}} \widebar{L}_\varPsi^2 - \frac{\alpha \gamma \lambda}{M} \Big] Q \\
&= \frac{1}{M} \Big[\frac{\widebar{L}_f^2 M}{2} \big( (\alpha+1){L_{\nabla g}} + \alpha \lambda \big)^2 - \alpha \gamma \big(\lambda  - 2 \widebar{L}_\varPsi^2 M{L_{\nabla g}} \big) \Big] Q \le 0.
\end{align*}
This verifies the first assertion.
Now, suppose $\dt{W}=0$. It follows from the last inequality that $Q=0$. Then, Proposition \ref{p:QG_L_delta_Q_beta} implies that
\[
\varPi_{\Tc_{\Bb}(\beta)}\big(-G_\beta \big) = 0
\]
for some $G_\beta \in \hat{\partial}{G}(\beta)$. Consequently, using the decomposition into the projections on polar cones,
we obtain
\[
-G_\beta = \varPi_{\Tc_{\Bb}(\beta)}\big(-G_\beta \big) + \varPi_{\Nc_{\Bb}(\beta)}\big(-G_\beta \big) =
\varPi_{\Nc_{\Bb}(\beta)}\big(-G_\beta \big).
\]
This means that the optimality condition \eqref{optimality} is satisfied.
\end{proof}
In order to prove the convergence of the discrete-time algorithm, we need an additional technical assumption (known as the \emph{Sard condition}).
\begin{description}
\item[(A8)] The set $\{G(\beta): \beta\in {\Bb}^*\}$ does not contain an interval of nonzero length.
\end{description}
\begin{theorem}
    \label{t:full-convergence}
    Assume conditions {\rm (A1)--(A8)} and \eqref{negative-der}. Then, except for a zero measure set $\varOmega_0\subset \varOmega$,  the sequence $\big\{w(z^k(\omega))\big\}_{k\in \Nb}$ is convergent, and every accumulation point of the sequence
$\big\{z^k(\omega)\big\}_{k\in \Nb}$ is an element of $\Zb^*$.
\end{theorem}
\begin{proof}
The proof follows the steps of the proofs of \cite[Thm. 3.20]{duchi2018stochastic} or \cite[Thm. 3.5]{majewski2018analysis} {  or \cite{nurminskii1972convergence}}.
We provide a short outline here, for the convenience of the Reader. For brevity, we omit the
argument $\omega$, but we stress that we consider a particular path of the method.

Define
\[
\underline{w} = \liminf_{k \to \infty^{\phantom{|}}} w(z^k), \qquad \overline{w} = \limsup_{k \to \infty}w(z^k).
\]
Suppose $\underline{w} < \overline{w}$. By Assumption (A7), a value ${v}\in
(\underline{w},\overline{w})$ exists,
such that ${v} \notin G(\Bb^*)$. Take any $u \in ({v},\overline{w})$.
Define $L_1=\min\big\{k : w(z^k) \le {v} \big\}$ and for
$j=1,2,\dots$
\begin{align*}
\ell_j &= \min\big\{k\ge L_j : w(z^k) > u \big\},\\
L_{j+1} &= \min\big\{k\ge \ell_j : w(z^k) \le {v} \big\}.\\
\intertext{Then we define}
k_j &= \max\big\{ k\in [L_j,\ell_j]: w(z^k) \le {v} \big\}, \quad j=1,2,\dots.
\end{align*}
Both $k_j\to \infty$ and $\ell_j\to \infty$, as $j\to \infty$.

Consider the sums
$T_j = \sum_{k=k_j}^{\ell_j-1} \tau_k$. Due to the equicontinuity of the functions
 $w(Z_j(\Cdot))$, a constant $T >0$ exists, such that $T_j\ge T$ for all sufficiently
large $j$. Set $s_j=t_{k_j}$ in \eqref{Z-shifted} and consider the interval $[0,T]$.
By Theorem \ref{t:diff-inc-beta}, a subsequence of the sequence $\{Z_j(\Cdot)\}$ is convergent to a solution $Z_\infty(\Cdot)$ of the differential
inclusion \eqref{diff-inc-beta}. For simplicity, we still denote it by $\{Z_j(\Cdot)\}$. As $W(\Cdot)$ is absolutely continuous,
\begin{equation}
\label{gradient-desc-cont}
w(Z_\infty(T)) = w(Z_\infty(0)) + \int_0^T  \dt{W}(t)  \;\D t.
\end{equation}
Since ${v} \notin w(\Zb^*)$, then  $Z_\infty(0) = \lim_{j\to \infty} z^{k_j} \notin \Zb^*$.
Theorem \ref{t:negative-der} then implies that $\dt{W}(0)<0$.
The generalized subdifferential $\hat{\partial} G(\Cdot)$ and the normal cone $\Nc_\Bb(\Cdot)$ are upper semicontinuous and closed-valued. The function $Q(\Cdot,\Cdot)$ is continuous.
Therefore, $\dt{W}(t)<0$
 in some interval $t\in [0,\bar{t}\,]$ of positive length.

 It follows from \eqref{gradient-desc-cont}
 that $w(Z_\infty(T)) < w(Z_\infty(0)) =v$.
 However, $Z_\infty(T)$ is the limit of the sequence $\{z^{m_j}\}$, where
\begin{equation}
\label{mj}
m_j = \min\bigg\{m \ge k_j: \sum_{k=k_j}^{m_j-1}\tau_k \ge T\bigg\}, \quad j=1,2,\dots.
\end{equation}
For all sufficiently large $j$ we  have $k_j \le m_j \le \ell_j$ and $w(z^{m_j}) > v$, which is a contradiction. Therefore,
$\underline{w} = \overline{w}$ and the sequence $\big\{w(z^k)\big\}$ has a limit.

Suppose a subsequence $\{z^{k_j}\}$ is convergent to a point $\bar{z}\notin \Zb^*$. Again, we
set $s_j = t_{k_j}$ in \eqref{Z-shifted}, consider the interval $[0,T]$ with $T>0$, and conclude that
a  sub-subsequence of $\{Z_j(\Cdot)\}$ exists, which is uniformly convergent to a solution $Z_\infty(\Cdot)$ of the
differential inclusion \eqref{diff-inc-beta}. Then the points $\{z^{m_j}\}$ are convergent to $Z_\infty(T)$ on this sub-subsequence,
where $m_j$ are defined as in \eqref{mj}.
As $Z_\infty(0)=\bar{z}\notin \Zb^*$, we obtain $Z_\infty(T)< Z_\infty(0)$. This would mean that a sub-subsequence
of $\big\{w(z^{m_j})\big\}$ has the limit $w(Z_\infty(T))< w(\bar{z})$, which contradicts the convergence of the sequence
$\big\{w(z^{k})\big\}$.
\end{proof}

It follows from the theorem that $\big\{Q(\beta^k,\theta^k)\big\} \to 0$ almost surely, the
sequence $\big\{G(\beta^k)\big\}$ is convergent as well,
and every accumulation point of the sequence  $\big\{\beta^k,\theta^k\big\}$ is in $\Zb^*$.

\begin{corollary}
Under Assumptions \textup{(A1)-(A8)}, the  method \eqref{method} is convergent if $\gamma > 2\widebar{L}_f^2 M {L_{\nabla g}}$.
\end{corollary}
\begin{proof}
Both $\alpha>0$ and $\lambda>0$ are arbitrary parameters of the Lyapunov function. Choosing $\alpha = {L_{\nabla g}} / ({L_{\nabla g}} + \lambda)$, we see from \eqref{negative-der} that it is sufficient that
$\gamma>\frac{2 \widebar{L}_f^2 M {L_{\nabla g}} ({L_{\nabla g}} + \lambda)}{\lambda - 2 \widebar{L}_\varPsi^2 M {L_{\nabla g}}}$. The parameter
$\lambda$ is not bounded from above. If we let $\lambda \to \infty$ and $\alpha \downarrow 0$ then the lower bound for $\gamma$ will converge to  $2\widebar{L}_f^2 M {L_{\nabla g}}$. Thus, for any $\lambda$ greater than this lower bound, the corresponding $\alpha$ and $\lambda$ can be found, such that \eqref{negative-der} is satisfied.
\end{proof}

\section{Numerical Illustration}
\label{s:5}

We revisit Example \ref{e:RL} in \S \ref{s:1}. We make further assumptions about the MDP $\{\Sc, \Ac, \Pb, r, \gamma \}$:
\begin{tightlist}{ii}
\item
The reward function  $r(s, a, s')$ is invariant with respect to $s' \in \Sc$; we shall write it as $r(s,a)$, with $r: \Sc \times \Ac \to \Rb$;
\item
The MDP $\{\Sc, \Ac, \Pb, r, \gamma \}$ is \emph{linear} with a feature map $\phi: \Sc \times \Ac \to \Rb^n$: there exist $n$ \emph{unknown} signed measures $\mub=(\mu^{(1)},\dots,\mu^{(n)})$ over $\Sc$ and an \emph{unknown} vector $\nub \in \Rb^n$, such that for any $(s, a) \in \Sc \times \Ac$, we have:
\begin{equation}
\label{a:LMDP}
\Pb(\Cdot|s, a) = \langle \phi(s, a), \mub(\Cdot) \rangle, \qquad r(s, a) = \langle \phi(s, a), \nub \rangle.
\end{equation}
\end{tightlist}
The Reader is referred to \cite{melo2007q,parr2008analysis,yang2019sample,jin2020provably} for the theory of linear MDPs and
relevant recent research in this area.

If we select the parametric family $\varPhi$ as the linear architecture model, $\varPhi(s, a, \beta) = \phi(s, a)^T \beta$, then for the function $f(\Cdot,\Cdot,\Cdot)$ defined in \eqref{f-g-MDP}, we have
\begin{equation}
\label{F-linear}
\begin{aligned}
F(X, \beta) &= \Eb\big[f(X, Y, \beta)\big|X\big]\\
&= \varPhi(s, a, \beta) - r(s, a) - \gamma \Eb\big[\varPhi(s', a', \beta)\big|s, a\big] \\
&= \phi(s, a)^T \beta - \phi(s, a)^T \nub - \gamma \sum_{s'} \Pb(s'|s, a) \Eb_{a' \sim \pi(s')}\big[\phi(s', a')^T \beta\big] \\
&= \phi(s, a)^T \beta - \phi(s, a)^T \nub - \gamma \sum_{s'} \phi(s, a)^T \mu(s') \sum_{a'} \pi(s')(a') \phi(s', a')^T \beta \\
&= \phi(s, a)^T \Big(\beta - \nub - \gamma \sum_{s'} \mu(s') \sum_{a'} \pi(s')(a') \phi(s', a')^T \beta\Big).
\end{aligned}
\end{equation}
The  double sum $\sum_{s'} \mu(s') \sum_{a'} \pi(s')(a') \phi(s', a')^T$
is constant as long as the policy $\pi$ is given, and thus the expression in the parentheses
in the last row of \eqref{F-linear} is a linear function of $\beta$. Despite its complication with unknown parameters and measures involved, the function \eqref{F-linear} can be treated as a linear combination of the features $\phi(s, a)$. Therefore, it is sufficient to have $\varPsi(X, \theta) = \phi(s, a)^T \theta$. We do not need to calculate how is $\theta$ linked to $\beta$; we only make sure that the model exists and is sufficiently rich.

We focus on a linear MDP with $|\Sc|=100$ states, $|\Ac|=50$ actions and discount $\gamma = 0.95$. We set the convex set $\Bb := \{\beta: \|\beta\| \le 10\}$ and $\Theta_R := \{\theta: \|\theta\| \le 1000\}$. In this problem, we use the Huber loss for the function $g(\Cdot)$, with the parameter {$\delta = 1$}:
\[
g(u) = \begin{cases}
    \frac{1}{2} \|u\|^2,& \text{if } \|u\| \le \delta\\
    \delta \|u\| - \frac{1}{2} \delta^2,              & \text{if } \|u\| > \delta.
\end{cases}
\]
It satisfies our assumptions for $g(\cdot)$.

We randomly generated $\phi, \mu, \nu$ in \eqref{a:LMDP} and trained the $(\beta, \theta)$ pair by the method \eqref{method}. {We selected $\tau_k = \frac{1}{1+k/1000}$ and $\gamma = 100$ and executed 5000 iterations. }

It is difficult to directly observe the progress of the method because it minimizes a function
whose values cannot be observed; we do not even have their unbiased statistical estimates.
The tracking error $Q$ cannot be observed as well.
Therefore, \textit{for the sole purpose of illustration}, we generated a
test sample $\Dc = \{(X^i, Y^i)\}_{1\le i \le |\Dc|}$ with $|\Dc| = 1000$ independent observations from the joint distribution, and we used it to evaluate the quality of the
successive iterates.

The first row in Figure \ref{fig:LMDP} depicts three quantities that can be directly observed at each iteration of the training process {$g(f(X^k, Y^k, \beta^k))$, $g(\varPsi(X^k, \theta^k))$}
and a biased estimate of the tracking error $Q^k \triangleq \break \frac{1}{2} \|f(X^k, Y^k, \beta^k) - \varPsi(X^k, \theta^k)\|^2$. The last two quantities have some downward trend, but it is very hard to judge the efficacy of the method from these observations.

Therefore, during the training process, given the current $(\beta, \theta) = (\beta^k, \theta^k)$, we calculated the following test metrics:
{
\begin{equation*}
\begin{aligned}
&\Eb^{\Dc}[g(f(X, Y, \beta^k))] \triangleq \frac{1}{|\Dc|} \sum_{(X^i, Y^i) \in \Dc} g(f(X^i, Y^i, \beta^k)),\\
&\Eb^{\Dc}[g(\varPsi(X, \theta^k))]  \triangleq \frac{1}{|\Dc|} \sum_{(X^i, Y^i) \in \Dc} g(\varPsi(X^i, \theta^k)),\\
&\Eb^{\Dc}[Q(\beta^k, \theta^k)]  \triangleq \frac{1}{2|\Dc|} \sum_{(X^i, Y^i) \in \Dc} \|f(X^i, Y^i, \beta^k) - \varPsi(X^i, \theta^k)\|^2.
\end{aligned}
\end{equation*}
}
To measure the optimality condition, we also estimated the average directions:
{
\begin{equation*}
\begin{aligned}
&\|\Eb^{\Dc}[\tilde{d}_{\beta^k}]\| \triangleq   \Big\|\frac{1}{|\Dc|} \sum_{i=1}^{|\Dc|} \Tilde{d}^i_\beta |_{\beta=\beta^k} \Big\|,\\
&\|\Eb^{\Dc}[\tilde{d}_{\theta^k}]\| \triangleq   \Big\|\frac{1}{|\Dc|} \sum_{i=1}^{|\Dc|} \Tilde{d}^i_\theta |_{\theta=\theta^k} \Big\|.
\end{aligned}
\end{equation*}
}
All subfigures, except {``Test $\Eb^{\Dc}[g(f(X, Y, \beta))]$''}, are shown in the logarithmic scale on the vertical axis.
We can see that { $\Eb^{\Dc}[g(\varPsi(X, \theta^k))]$} is converging to its minimum for the Huber loss $g(\cdot)$, which is zero due to the linearity of the model; { $\Eb^{\Dc}[Q(\beta^k, \theta^k)]$ and $\|\Eb^{\Dc}[\tilde{d}_{\theta^k}]\|$} are converging to $0$ as well. These observations indicate that $\varPsi(\Cdot,\theta)$ tracks $F(\Cdot,\beta)$ closely and the original composite loss function is indeed minimized.

\begin{figure}[t!]\captionsetup[sub]{font=footnotesize}
    \centering
    \begin{subfigure}[]{0.32\textwidth}
        \centering
        \includegraphics[width=4.1cm]{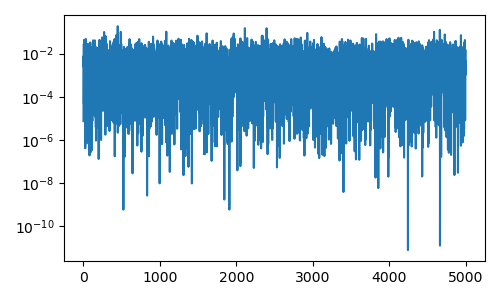}
        \caption{Train $g(f(X^k, Y^k, \beta^k))$}
    \end{subfigure}%
    \begin{subfigure}[]{0.32\textwidth}
        \centering
        \includegraphics[width=4.1cm]{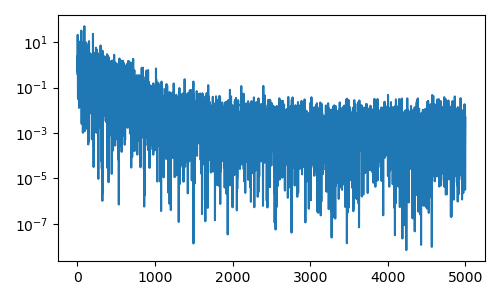}
        \caption{Train $g(\varPsi(X^k, \theta^k))$}
    \end{subfigure}
    \begin{subfigure}[]{0.32\textwidth}
        \centering
        \includegraphics[width=4.1cm]{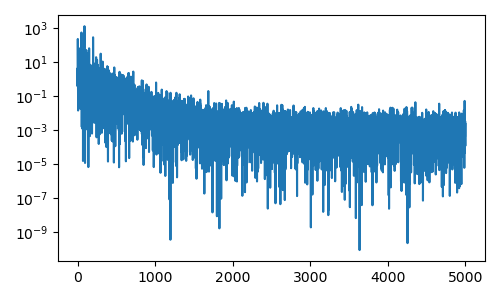}
        \caption{Train $Q^k$}
    \end{subfigure}

    \begin{subfigure}[]{0.32\textwidth}
        \centering
        \includegraphics[width=4.1cm]{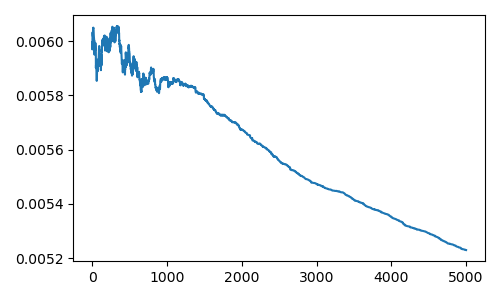}
        \caption{Test $\Eb^{\Dc}[g(f(X, Y, \beta^k))]$}
    \end{subfigure}%
    \begin{subfigure}[]{0.32\textwidth}
        \centering
        \includegraphics[width=4.1cm]{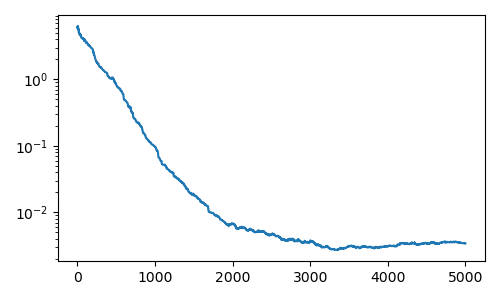}
        \caption{Test $\Eb^{\Dc}[g(\varPsi(X, \theta^k))]$}
    \end{subfigure}
    \begin{subfigure}[]{0.32\textwidth}
        \centering
        \includegraphics[width=4.1cm]{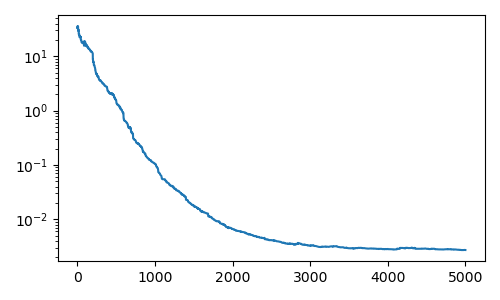}
        \caption{Test $\Eb^{\Dc}[Q(\beta^k, \theta^k)]$}
    \end{subfigure}

    \begin{subfigure}[]{0.32\textwidth}
        \centering
        \includegraphics[width=4.1cm]{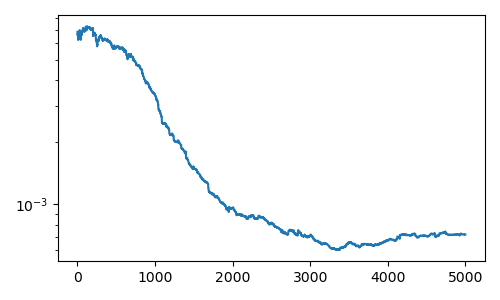}
        \caption{Test $\|\Eb^{\Dc}[\tilde{d}_{\beta^k}]\|$}
    \end{subfigure}%
    \begin{subfigure}[]{0.32\textwidth}
        \centering
        \includegraphics[width=4.1cm]{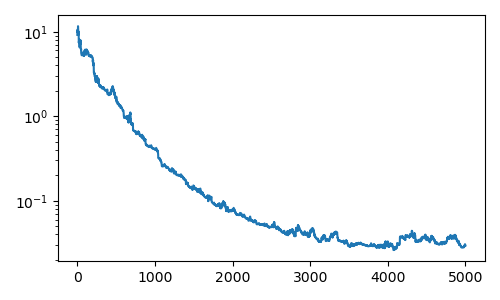}
        \caption{Test $\|\Eb^{\Dc}[\tilde{d}_{\theta^k}]\|$}
    \end{subfigure}

    \caption{The progress of the method on the training set (the top three subgraphs) and on the test set (the bottom five subgraphs).}
\label{fig:LMDP}
\end{figure}

\section{Conclusions and Future Research}

We consider conditional stochastic optimization problems, where two major difficulties
are present:
\begin{tightlist}{ii}
    \item No unbiased estimates of the function that is minimized or its subgradients are available because sampling from the conditional distribution is not possible;
    \item The function is differentiable in the generalized sense only and may be not convex.\vspace{1ex}
\end{tightlist}
For these problems, we propose a new stochastic method, which involves an auxiliary functional model, whose purpose is to track the ``invisible'' conditional expectation in the functional space. The algorithm is not a stochastic subgradient method, because its directions, even in the
expected value sense, are not subgradients of any function.

The convergence analysis of the method involves several novel ideas: a stochastic version of
the Bregman distance, and a stochastic {\L}ojasiewicz condition, and follows the differential inclusion method with a specially designed Lyapunov function.

We hope the new ideas we introduce here will find their way to the theory and methods for other stochastic optimization problems involving compositions of functional models.

{  Many different research questions may be pursued on the foundation
provided by the present paper. Versions of other, more complex stochastic subgradient algorithms
for nonconvex and nonsmooth problems may be developed, such as mirror descent \cite{ghadimi2013stochastic,davis2019stochastic}, dual averaging \cite{Rusz2019}, \textit{etc.}; each of them deserves a separate study.
{ A major challenge is
\emph{contextual bi-level optimization}, as in \cite{hu2024contextual}. Our conditional model is a special case of it, with the
conditional expectation $\Eb [f(X,Y,\beta) \,|\, X]$ interpreted as the minimizer
$z^*(X,\beta)$ in the second-level problem:
\[
\min_{z} \;\Eb\big[ \big(f(X,Y,\beta) - z\big)^2 \,\big|\, X \big].
\]
We do not know whether it is possible to replace the least-squares objective in this problem with a broader class of functions and still be able to solve the
upper-level problem:
\[
\min_{\beta \in {\Bb}} G(\beta) \triangleq \Eb \Big\{g\big(X,\beta,z^*(X,\beta)\big)\Big\}.
\]
Such a generalization would open the door to many applications discussed in \cite{hu2024contextual}.
}

A natural question arises how to deal with the \textit{misspecified setting}, when the
parametric family of functional models $\big\{ \varPsi(X, \theta)\big\}_{\theta\in \Theta}$ is not rich enough to reproduce the conditional expectations $\big\{ F(X,\beta)\big\}_{\beta\in \Bb}$. It is a major research challenge. In the machine learning literature, some special cases have been
successfully analyzed in the linear function approximation case \cite{ghosh2017misspecified,jin2020provably}.
Our situation is more complicated, and we plan to address this challenge in our future research. A positive aspect is that the method \eqref{method}--\eqref{dzk} remains well-defined in this case and appears to produce good approximations of the solution.
However, without thorough analysis, it is too early to speculate about its virtues.
}

\bibliographystyle{plain}

\appendix

\section{Generalized differentiability of functions}

Norkin \cite{norkin1980generalized} introduced the following class of functions.
\begin{definition}
\label{d:Norkin}
A function $f:\Rb^n\to\Rb$ is  \emph{differentiable in a generalized sense at a point} $x\in \Rb^n$,
if an open set $U\subset \Rb^n$ containing $x$,
and a nonempty, convex, compact valued, and upper semicontinuous multifunction
$\hat{\partial}f: U \rightrightarrows \Rb^n$ exist, such that for all $y\in U$ and all $g \in \hat{\partial}f(y)$ the following equation is true:
\[
f(y) = f(x) + \langle g(y), y-x \rangle + o(x,y,g),
\]
with
\[
\lim_{y\to x} \sup_{g\in \hat{\partial}f(y)} \frac{o(x,y,g)}{\|y-x\|}=0.
\]
The set $\hat{\partial}f(y)$ is  the \emph{generalized subdifferential} of $f$ at $y$.
If a function is differentiable in a generalized sense at every  $x \in \Rb^n$ with the same generalized subdifferential mapping $\hat{\partial}f:\Rb^n\rightrightarrows \Rb^n$, we call it \emph{differentiable in a generalized sense}.

A function $f:\Rb^n\to\Rb^m$ is  differentiable in a generalized sense, if each of its component functions, $f_i:\Rb^n\to\Rb$, $i=1,\dots,m$, has this property.
\end{definition}

Compositions of generalized differentiable functions are crucial in our analysis.
\begin{theorem}
\label{t:composition}
If $h:\Rb^m \to \Rb$ and $f_i:\Rb^n\to \Rb$, $i=1,\dots,m$, are differentiable in a generalized sense, then the composition
$\varPsi(x) = h\big( f_1(x),\dots,f_m(x)\big)$
is differentiable in a generalized sense, and at any point $x\in \Rb^n$ we can define the generalized subdifferential of $\varPsi$ as follows:
\begin{multline}
\label{sub-comp}
\hat{\partial}\varPsi(x) = \text{\rm conv} \big\{ g\in \Rb^n: g = \begin{bmatrix} f_{1x} & \cdots & f_{mx}\end{bmatrix} h_f,\\ \text{ with } h_f\in \hat{\partial}{h}\big(f_1(x),\dots,f_m(x)\big) \text{ and } f_{jx}\in \hat{\partial}{f_j}(x),\ j=1,\dots,m\big\}.
\end{multline}
\end{theorem}
Even if we take $\hat{\partial}{h}(\Cdot)=\partial h(\Cdot)$ and $\hat{\partial}{f_j}(\Cdot)=\partial\! f_j(\Cdot)$, $j=1,\dots,m$, we may obtain
$\hat{\partial}\varPsi(\Cdot) \ne \partial \varPsi(\Cdot)$,
but $\hat{\partial}\varPsi$ defined above satisfies Definition \ref{d:Norkin}. 

For stochastic optimization, essential is the closure of the class functions differentiable in a generalized sense with respect to
expectation.
\begin{theorem}
\label{t:expected-generalized}
Suppose $(\varOmega,\Fc,P)$ is a probability space and a function $f:\Rb^n\times \varOmega \to \Rb$ is differentiable in a generalized sense with respect to $x$ for all $\omega\in \varOmega$, and integrable with respect to $\omega$ for all $x\in \Rb^n$. Let $\hat{\partial}f: \Rb^n \times \varOmega
\rightrightarrows \Rb^n$ be a multifunction, which is measurable with respect to $\omega$ for all $x\in \Rb^n$, and which is a generalized subdifferential mapping of $f(\Cdot,\omega)$ for all $\omega\in \varOmega$. If for every compact set $K\subset \Rb^n$ an integrable function
$L_K:\varOmega\to \Rb$ exists, such that
$\sup_{x\in K}\sup_{g\in \hat{\partial}f(x,\omega)}\|g\| \le L_K(\omega)$, $\omega \in \varOmega$,
then the function
\[
F(x) = \int_\varOmega f(x,\omega)\;P(d\omega),\quad x\in \Rb^n,
\]
is differentiable in a generalized sense, and the multifunction
\[
\hat{\partial}F(x) = \int_\varOmega \hat{\partial}f(x,\omega)\;P(d\omega),\quad x\in \Rb^n,
\]
is its generalized subdifferential mapping.
\end{theorem}

An essential step in the analysis of stochastic recursive algorithms by the differential inclusion method is the
\emph{chain rule on a path} (see \cite{davis2019stochastic}
and the references therein). For an absolutely continuous function $p:[0,\infty)\to\Rb^n$ we denote
by $\dt{p}(\Cdot)$ its weak derivative: a measurable function such that
\[
p(t) = p(0) + \int_0^t \dt{p}(s)\;ds,\quad \forall \;  t \ge 0.
\]
The following theorem has been recently proved in \cite{bolte2023subgradient}.
\begin{theorem}
\label{t:chain}
If the function $f:\Rb^n \to \Rb$ is differentiable in a generalized sense, then it admits the chain rule on absolutely continuous paths, that is
\begin{equation}
\label{chain-path}
f(p(T))- f(p(0)) = \int_0^T   g(p(t)) \, \dt{p}(t)  \;dt,
\end{equation}
for any absolutely continuous path $p:[0,\infty)\to \Rb^n$, all selections $g(\Cdot) \in \hat{\partial}\!f(\Cdot)$, and all $T>0$.
\end{theorem}
Formula \eqref{chain-path} was previously known for convex functions \cite{brezis1971monotonicity}, for
subdifferentially regular locally Lipschitz functions
and Whitney $\mathcal{C}^1$-stratifiable locally Lipschitz functions \cite{drusvyatskiy2015curves}, and for generalized differentiable functions
on generalized differentiable paths \cite{Rusz2019}.

\end{document}